\documentclass[oneside,english]{amsart}
\usepackage[T1]{fontenc}
\usepackage[latin9]{inputenc}
\usepackage{color}
\usepackage{float}
\usepackage{amstext}
\usepackage{amsthm}
\usepackage{amssymb}
\usepackage{graphicx}

\makeatletter

\providecommand{\tabularnewline}{\\}
\newcommand{\lyxdot}{.}

\numberwithin{equation}{section}
\numberwithin{figure}{section}
\theoremstyle{plain}
\newtheorem*{thm*}{\protect\theoremname}
\theoremstyle{plain}
\newtheorem{thm}{\protect\theoremname}
\theoremstyle{definition}
\newtheorem{defn}[thm]{\protect\definitionname}
\theoremstyle{definition}
\newtheorem{example}[thm]{\protect\examplename}
\theoremstyle{plain}
\newtheorem{prop}[thm]{\protect\propositionname}

\usepackage{mathtools}

\makeatother

\usepackage{babel}
\providecommand{\definitionname}{Definition}
\providecommand{\examplename}{Example}
\providecommand{\propositionname}{Proposition}
\providecommand{\theoremname}{Theorem}

\begin{document}
\global\long\def\Aacces#1#2#3{\prescript{}{#1}{\mathbf{\mathbb{A}}^{#3}}_{#2}^{\ }}%

\global\long\def\Aaccess#1#2#3{\prescript{}{#1}{\mathbf{\mathbb{A}}}_{#2}^{#3}}%

\global\long\def\Tun#1#2#3{\prescript{#3}{#2}{\Theta}_{#1}^{01}}%

\global\long\def\Teu#1#2#3#4{\prescript{#4}{#3}{\Theta}_{#1,#2}^{02}}%

\global\long\def\Tnu#1#2#3#4{\prescript{#4}{#3}{\Theta}_{#1,#2}^{11}}%

\global\long\def\parent#1{\mathfrak{p}_{#1}}%

\global\long\def\voisin#1{\mathfrak{v}_{#1}}%

\global\long\def\AA{\mathbb{A}}%

\global\long\def\pair#1#2#3{\underset{#3}{\left\lfloor \begin{array}{c}
#1\\
#2
\end{array}\right.}}%

\title{The Saito vector field of a germ of complex plane curve.}
\author{Yohann Genzmer}
\begin{abstract}
In this article, we prove that an algorithm introduced by the author
in a previous work and giving the generic dimension of the moduli
space of a germ of curve in the complex plane that is the union of
smooth curves, can be used identically to find this dimension for
any kind of germ of plane curve.
\end{abstract}

\maketitle

\section{Introduction.}

\subsection*{Saito vector field.}

Let $C$ be a germ of curve in the complex plane $\left(\mathbb{C}^{2},0\right).$
The curve is given by the $0-$level of an holomorphic reduced function
$f:\left(\mathbb{C}^{2},0\right)\to\left(\mathbb{C},0\right)$. The
Saito module is by definition, the module $\textup{Der}\left(\log C\right)$
of germ of vector fields $X$ tangent to $C,$i.e, 
\[
X\cdot f\in\left(f\right),
\]
where $X\cdot f$ is the result of the application of $X$ to $f$
seen as the derivation. Naturally, this module contains a lot of informations
on the curve. Of particular interest are the valuations of its elements
and the so called Saito number of $C$, that is, 
\[
\mathfrak{s}_{C}=\min_{X\in\textup{Der}\left(\log C\right)}\nu\left(X\right).
\]

This number is clearly an analytical invariant of $C$. Nevertheless,
it is not a topological invariant of the curve : for instance, if
$C$ is the union of $5$ straigth lines then the radial vector field
is tangent and thus, $\mathfrak{s}_{C}=1.$ But for a generic perturbation
of $C$ that leaves invariant its topological type, it can be seen
that $\mathfrak{s}_{C}=2.$ In \cite{genzmer2020saito}, we proved
the following upper bound that holds along the topological class $\textup{Top\ensuremath{\left(C\right)} }$of
$C$
\begin{equation}
\max_{c\in\textup{Top}\left(C\right)}\mathfrak{s}_{C}=\left\lfloor \frac{\nu\left(C\right)}{2}\right\rfloor \textup{ or }\frac{\nu\left(C\right)}{2}-1\label{eq:uppe}
\end{equation}
depending on $C$ being \emph{radial} or not. A vector field $X$
tangent to $C$ whose valuation reaches the Saito number is said \emph{optimal.
}In the present article, we are interesting in the topology of an
optimal vector field for a curve $c\in\textup{Top}\left(C\right)$
which realizes the upper bound (\ref{eq:uppe}). Such a generic vector
field will be said \emph{Saito} for $C.$ As it is highlighted in
\cite{moduligenz}, some invariants associated to the topology of
the Saito vector field of $C$ allows us to produce an algorithm for
the computation of the generic dimension of the moduli space of $C.$ 

\subsection*{Generic dimension of the moduli space of $C.$}

The problem of the determination of the number of moduli of a germ
of complex plane curve was addressed by Oscar Zariski in his famous
notes \cite{zariski}, where he focused on the case of a curve with
only one irreducible component. The\emph{ number of moduli} refers
to the number of analytical invariants that remain once the topological
class of $S$ is given. The topological classification of an irreducible
curve $S$ is well known and relies on a semi-group of integers extensively
studied by Zariski himself in the 70s. However, at this time, even
in the case of an irreducible curve, the analytical classification
was a widely opened question. Since then, a lot of progress has been
made, and, up to our knowledge, the initial problem has been considerably
investigated in the combination of the works of A. Hefez, M. Hernandes
and M. E. Hernandes \cite{MR2509045,MR2781209,MR2996882,hernandes2023analytic}. 

In \cite{moduligenz}, the author constructed an algorithm to compute
the generic dimension of the moduli space of a germ of complex curve
$C$ and proved that this algorithm yields the desired dimension under
the assumption that $C$ is an union of smooth curves. In \cite{YoyoBMS},
it was established that the same algorithm still works if $C$ is
irreducible. The goal of this article is to prove that this algorithm
provides the expected dimension in any case. 

If $C$ is a germ of curve in the complex plane, we consider $\textup{Mod}\left(C\right)$
the quotient of the set of curves topologically equivalent to $C$
by the action of the local conjugacies $\textup{Diff}\left(\mathbb{C}^{2},0\right)$.
In \cite{genzmer2020saito} following ideas of Ebey \cite{MR0176983},
we identify $\textup{Mod}\left(C\right)$ to the quotient of a constructible
subset of some finite dimension complex vector space by the action
of an algebraic connected group of finite dimension. This identification
allows us to consider the generic dimension of this quotient. The
mention algorithm intends to compute this dimension from primitive
topological invariants of $C.$ The approach is as follows :
\begin{enumerate}
\item for $c\in\textup{Mod}\left(C\right)$, $c$ being generic, the dimension
of the cohomological space $\textup{H}^{1}\left(D,\Theta\right)$
where $D$ is the exceptional divisor of the desingularization process
of $c$ and $\Theta$ the sheaf over $D$ of vector field tangent
to $D$ and the strict transform of $c,$ is equal to the generic
dimension of $\textup{Mod}\left(C\right).$
\item for any curve $c,$ the module $\textup{Der}\left(\log c\right)$
known since K. Saito \cite{MR586450} to be freely generated by two
vector fields $X_{1}$ and $X_{2}$. One of them reaches the Saito
number of $c.$ 
\item when $c$ is generic, the topology of a generic vector field reaching
the Saito number of $c$ can be described and is somehow unique. 
\item from this description, one can compute the dimension of $\textup{H}^{1}\left(D_{0},\Theta\right)$
where $D_{0}$ is the exceptional divisor of the first blowing-up
appearing in the process of desingularization of $c,$ initiating
a computation of the dimension of the whole $\textup{H}^{1}\left(D,\Theta\right)$
inductively along the process of blowing-ups. 
\end{enumerate}
The main result of this article focuses on the claim (3) among these
above and can be expressed as follows :
\begin{thm*}
Let $C$ be a germ of curve in $\left(\mathbb{C}^{2},0\right)$ generic
in its moduli space. Let $X$ be Saito for $C$. Let $\AA$ be the
dual tree of the desingularization process $E$ of $C.$ We number
a vertex $s$ of $\AA$ by the number of singular points of the strict
transform $X^{E}$ along $s.$ We color a vertex in white if $s$
is invariant by $X^{E}$, otherwise we color it in black. Then, the
colored numbered tree $\AA$ depends only on the topological class
of $C.$ 
\end{thm*}
The article is divided in two sections : the first can be red independently.
It focuses on a combinatorial result on trees. However, this result
will be a key element to describe the topology of the generic vector
field tangent to a generic curve $C,$ that is the main goal of the
second section. 

\section{Saito dicriticity of an ordered numbered tree.}

Let $\AA$ be a tree. We can endow $\AA$ with a partial order $\leq$
defined following an inductive description of $\AA$ : starting from
a single vertex $r$, if $\left(\AA,\leq\right)$ is defined, one
can add a vertex to $\AA$ following one of the two next rules
\begin{enumerate}
\item a vertex $s$ and an edge from $s$ to a single vertex $c$ are added
to $\AA$. The order $\leq$ is extended to $\AA\cup\left\{ s\right\} $
setting $s\geq c.$
\item a vertex $s$ is added to $\AA$ deleting an egde from $c$ to $c^{\prime}$
and adding two edges from $s$ to $c$ and from s to $c^{\prime}$.
The order is extended setting $s\geq c$ and $s\geq c^{\prime}.$ 
\end{enumerate}
\begin{figure}
\begin{centering}
\includegraphics[width=0.7\textwidth]{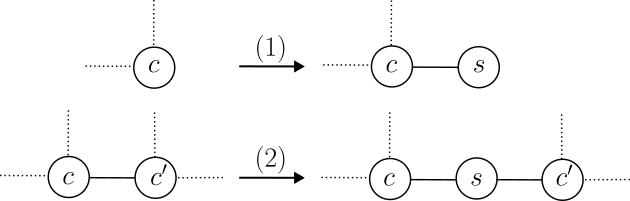}
\par\end{centering}
\caption{Inductive construction of the partial order $\protect\leq$ on $\protect\AA$.}

\end{figure}

The vertex $r$ is the minimal element of $\left(\AA,\leq\right)$
and is called the \emph{root} of $\AA.$ In the sequel, in general,
we will make no distinction between $\mathbb{A}$ and the set of vertices
of $\AA.$ 

We will denote by $n=\left(n_{c}\right)_{c\in\mathbb{A}}$ a numbering
of the vertices of $\mathbb{A}$ by non negative integers. We will
also consider an element $\Delta=\left(\Delta_{c}\right)_{c\in\mathcal{\mathbb{A}}}$
in $\left\{ 0,1\right\} ^{\mathbb{A}}.$ The latter is called \emph{a
dicriticity for $\mathbb{A}.$} It induces a coloring of the tree
$\AA$: if $\Delta_{c}=1$ the vertex $c$ is colored in white, if
not, it is colored in black. 

Finally, the notation $\pair abc$ stands for the the following :
\[
\pair abc=\left\{ \begin{array}{cl}
a & \textup{if }c\textup{ is even}\\
b & \textup{else}
\end{array}\right.
\]

\begin{defn}
\label{def1}In what follows, $c$ stands for a vertex of $\AA$. 
\begin{enumerate}
\item We denote by $\parent c$ \emph{the set of parents of $c$} that is
the set of predecessors of $c$ for the partial order $\leq$. Notice
that $\parent r=\emptyset$ ; in any other cases, $\parent c$ contains
one or two elements. 
\end{enumerate}
\end{defn}

\begin{enumerate}
\item Following \cite{MR2107253}, fixing a numbering $\left\{ 1,\ldots,N\right\} $
of the vertices such that $i\in\parent j\Longrightarrow i\leq j$,
we consider the proximity matrix $\mathbb{P}$ of $\left(\AA,\leq\right)$
defined by 
\[
\begin{array}{cl}
\mathbb{P}_{i,i} & =1\\
\mathbb{P}_{i,j} & =-1\textup{ if }i\in\parent j\\
\mathbb{P}_{i,j} & =0\textup{ else.}
\end{array}.
\]
It is an upper triangular invertible matrix.
\item If $n$ is numbering of $\mathbb{A}$, then $c\cdot n$ is the numbering
defined by 
\[
\left(c\cdot n\right)_{c}=n_{c}+1
\]
 and $\left(c\cdot n\right)_{c^{\prime}}=n_{c^{\prime}}$ if $c^{\prime}\neq c.$
\item In what follows, $\mathfrak{v_{c}}$ denotes \emph{the set of neighbours
of $c$} in $\AA$, that is the set of vertices of $\AA$ connected
to $c$.
\item We called \emph{the multiplicity of $c$} in $\AA$ the positive integer,
denoted by $\rho_{c}$ obtained inductively as follows : $\rho_{r}=1,$
and if $c\neq r$ then 
\[
\rho_{c}=\sum_{c^{\prime}\in\parent c}\rho_{c^{\prime}}.
\]
\item We denote by $\delta_{c}$ the number of parents $c^{\prime}$ of
$c$ such that $\Delta_{c^{\prime}}=1.$ In particular, this number
depends not only on $\AA$ but also on a dicriticity $\Delta$. 
\item We called\emph{ the valuation of $c$} the number denoted by $\nu_{c}^{n}$
and defined by the matrix relation 
\[
\left(\nu_{c}^{n}\right)_{c\in\mathbb{A}}=\mathbb{P}^{-1}\left(n_{c}\right)_{c\in\mathbb{A}}.
\]
In particular, from \cite{MR2107253}, it follows that 
\[
\nu_{r}^{n}=\sum_{c\in\AA}\rho_{c}n_{c}
\]
\item The \emph{square index} of $c$ is defined by 
\[
\square_{c}=\frac{\delta_{c}}{2}-\pair{\Delta_{c}}{\frac{1}{2}}{\nu_{c}^{n}-\delta_{c}}.
\]
\item If $c$ and $c^{\prime}$ belong to $\AA,$ we defined\emph{ the acces
tree} from $c$ to $c^{\prime}$ the minimal subgraph of $\AA$ that
from $c$, leads to $c^{\prime}$ respecting the order $\leq.$ It
is denoted by 
\[
\Aacces c{c^{\prime}}{}.
\]
If $c=r$ is the root of $\AA$, then it is simply denoted by 
\[
\Aacces{}{c^{\prime}}{}.
\]
If $\rho_{c^{\prime}}=1$ then the access tree $\Aacces{}{c^{\prime}}{}$
is a totally ordered linear chain of vertices whose multiplicities
are equal to $1$ as in Figure \ref{fig:Acces-tree-from}. The proximity
matrix of this sub-graph is written
\[
\left(\begin{array}{cccc}
1 & -1\\
 & 1 & -1\\
 &  & \ddots & -1\\
 &  &  & 1
\end{array}\right)
\]
\begin{figure}[H]
\begin{centering}
\includegraphics[scale=0.3]{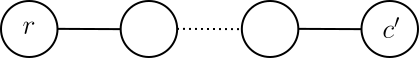}
\par\end{centering}
\caption{\label{fig:Acces-tree-from}Acces tree from $r$ to $c^{\prime}$
with $\rho_{c^{'}}=1.$}
\end{figure}
\item \label{configuration}We denote by $\epsilon=\left(\epsilon_{c}\right)_{c\in\AA}$
the family of integers defined by the following matrix relation 
\[
\left(\epsilon_{c}\right)_{c\in\AA}=\mathbb{P}\left(\frac{1}{2}\left(\nu_{c}^{n}\right)_{c\in\AA}-\left(\square_{c}\right)_{c\in\mathbb{A}}\right)=\left(\frac{n_{c}}{2}\right)_{c\in\AA}-\mathbb{P}\left(\left(\square_{c}\right)_{c\in\mathbb{A}}\right).
\]
 This uple of integers is called \emph{the configuration associated
to the dicriticity $\Delta.$}
\end{enumerate}
\begin{example}
Let us consider the tree represented in Figure (\ref{fig:Un.exemple}).

\begin{figure}
\begin{centering}
\includegraphics[width=0.7\linewidth]{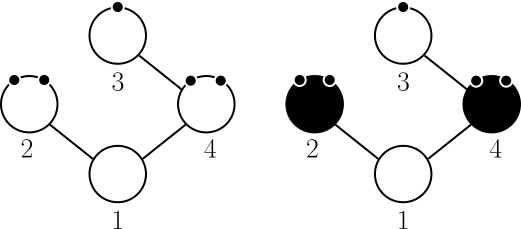}
\par\end{centering}
\caption{\label{fig:Un.exemple}An ordered tree, its numbering and its dicriticity.}

\end{figure}

In this example, the proximity matrix is written 
\[
\mathbb{P}=\left(\begin{array}{cccc}
1 & -1 & -1 & -1\\
0 & 1 & 0 & 0\\
0 & 0 & 1 & -1\\
0 & 0 & 0 & 1
\end{array}\right).
\]
The numbering is $n=\left(0,2,1,2\right)$ and is represented in Figure
(\ref{fig:Un.exemple}) by dots attached to the vertices. We have
\[
\begin{array}{l}
\parent 1=\emptyset,\ \parent 2=\left\{ 1\right\} ,\ \parent 3=\left\{ 1\right\} ,\ \parent 4=\left\{ 1,3\right\} \\
\voisin 1=\left\{ 2,4\right\} ,\ \voisin 2=\left\{ 1\right\} ,\ \voisin 4=\left\{ 1,3\right\} ,\ \voisin 3=\left\{ 4\right\} .
\end{array}
\]
The partial order induced on $\AA$ by the rules of construction is
 
\[
1\leq2,\ 1\leq3\leq4.
\]

The multiplicities are 
\[
\rho_{1}=1,\ \rho_{2}=1,\ \rho_{3}=1,\ \rho_{4}=2.
\]

Given the numbering $n$, the valuations are 
\[
\left(\begin{array}{c}
\nu_{1}\\
\nu_{2}\\
\nu_{3}\\
\nu_{4}
\end{array}\right)=\mathbb{P}^{-1}\left(\begin{array}{c}
0\\
2\\
1\\
2
\end{array}\right)=\left(\begin{array}{cccc}
1 & 1 & 1 & 2\\
0 & 1 & 0 & 0\\
0 & 0 & 1 & 1\\
0 & 0 & 0 & 1
\end{array}\right)\left(\begin{array}{c}
0\\
2\\
1\\
2
\end{array}\right)=\left(\begin{array}{c}
7\\
2\\
3\\
2
\end{array}\right).
\]

Now, assuming the dicriticity is $\Delta=\left(1,0,1,0\right)$ as
in the figure, we obtain 
\[
\delta_{1}=0,\ \delta_{2}=1,\ \delta_{3}=1,\ \delta_{4}=2
\]
and 
\[
\begin{array}{c}
\square_{1}=\frac{0}{2}-\pair 1{\frac{1}{2}}{7-0}=-\frac{1}{2},\ \square_{2}=\frac{1}{2}-\pair 0{\frac{1}{2}}{2-1}=0\\
\square_{3}=\frac{1}{2}-\pair 1{\frac{1}{2}}{3-1}=-\frac{1}{2},\ \square_{4}=\frac{2}{2}-\pair 0{\frac{1}{2}}{2-2}=1.
\end{array}
\]
Finally, the configuration $\epsilon$ is computed as follows 
\[
\left(\begin{array}{c}
\epsilon_{1}\\
\epsilon_{2}\\
\epsilon_{3}\\
\epsilon_{4}
\end{array}\right)=\frac{1}{2}\left(\begin{array}{c}
0\\
2\\
1\\
2
\end{array}\right)-\left(\begin{array}{cccc}
1 & -1 & -1 & -1\\
0 & 1 & 0 & 0\\
0 & 0 & 1 & -1\\
0 & 0 & 0 & 1
\end{array}\right)\left(\begin{array}{c}
-\frac{1}{2}\\
0\\
-\frac{1}{2}\\
1
\end{array}\right)=\left(\begin{array}{c}
1\\
1\\
2\\
0
\end{array}\right).
\]
\end{example}

As it can be seen in the previous example, from any dicriticity $\Delta$,
one can compute a configuration $\epsilon$ following Definition \ref{def1}.\ref{configuration}.
However, adding some constraints yields a unicity type result. 
\begin{thm}
\label{Theoreme.Un}Consider a numbering $n$ of $\AA$. 
\begin{enumerate}
\item \label{Existence.Unicite.Saito}There exists a unique dicriticity,
denoted by $\Delta^{n}=\left(\Delta_{c}^{n}\right)_{c\in\AA}$ such
that the associated configuration $\epsilon^{n}$ satisfies the following
relations
\begin{enumerate}
\item if $\Delta_{c}^{n}=0,$ then $\epsilon_{c}^{n}\geq2-\sum_{c^{\prime}\in\voisin c}\Delta_{c^{\prime}}^{n}$
\item if $\Delta_{c}^{n}=1$, then $\epsilon_{c}^{n}\geq n_{c}.$
\end{enumerate}
Such a dicriticity is said \emph{admissible} and is called \emph{the
Saito dicriticity of $\AA$ numbered by $n.$} The exponent $n$ appearing
on any data in the sequel will mean that these datas are associated
to the Saito dicriticity for a given numbering $n$.
\item \label{propriete.theta1}We define by $\Tun c{\mathbb{A}}n$ the following
invariant 
\[
\Tun c{\mathbb{A}}n=\sum_{s\in\Aacces{}c{}}\square_{s}^{n}+\square_{s}^{c\cdot n}.
\]
If $\rho_{c}=1$ then we obtain
\begin{equation}
\Tun c{\mathbb{A}}n=-\Delta_{c_{1}}^{n}-\frac{\left|\Aacces{}{c_{1}}{}\right|}{2}\label{Valeur de Theta1}
\end{equation}
where $\left|\star\right|$ denotes the number of vertices in the
subtree $\star.$
\item \label{propriete.theta2}Let $c_{0},c_{1}\in\mathbb{A}$. We define
by $\Teu{c_{0}}{c_{1}}{\mathbb{A}}n$ and $\Tnu{c_{0}}{c_{1}}{\mathbb{A}}n$
the following invariants 
\begin{align*}
\Teu{c_{0}}{c_{1}}{\mathbb{A}}n & =\sum_{s\in\Aacces{}{c_{1}}{}}\square_{s}^{n}+\square_{s}^{c_{1}\cdot c_{0}\cdot n}\\
\Tnu{c_{0}}{c_{1}}{\mathbb{A}}n & =\sum_{s\in\Aacces{}{c_{1}}{}}\square_{s}^{c_{0}\cdot n}+\square_{s}^{c_{1}\cdot n}.
\end{align*}
If $c_{0}$ and $c_{1}$ satisfy both $\rho_{c_{0}}=\rho_{c_{1}}=1$
then we have
\begin{align}
\Teu{c_{0}}{c_{1}}{\mathbb{A}}n & =\pm\frac{1}{2}-\Delta_{c_{1}}^{n}-\frac{\left|\Aacces{}{c_{1}}{}\right|}{2}\label{Valeur de Theta2}\\
\Tnu{c_{0}}{c_{1}}{\mathbb{A}}n & =\pm\frac{1}{2}-\Delta_{c_{1}}^{n}-\frac{\left|\Aacces{}{c_{1}}{}\right|}{2}\label{Valeur de Theta11}
\end{align}
\item \label{Mixed.branch}Let $c$ be a vertex of $\AA$ such that $\rho_{c}=1.$
We say that\emph{ }the access tree $\AA_{c}$ \emph{starts with a
mixed branch} if there exists a vertex $m_{c}\in\mathbb{A}_{c}$ maximal
for this property such that for any $s\in\mathbb{A}_{m_{c}}$, one
has 
\[
\Delta_{s}^{n}+\Delta_{s}^{c\cdot n}=1.
\]
It may happen that $m_{c}=c.$ If not, let us denote by $m_{c}^{+}$
the vertex of $\AA$ which succeeds $m_{c}$ in $\AA_{c}$. Depending
on the type the mixed branch, Table (\ref{Parit=0000E9 de nu}) presents
some properties of the valuations $\nu_{r}$ and $\nu_{m_{c}^{+}}$.
In this table, a picture such as \includegraphics[viewport=0bp 26.1811bp 299bp 95bp,width=0.07\paperwidth]{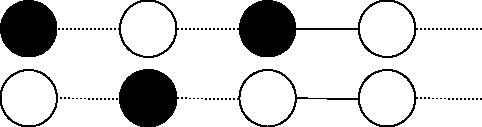}
represents the two Saito dicriticities along $\AA_{c}$ respectively,
above for the numbering $c$ and below for the numbering $c\cdot n$.
\\
Besides, if $m_{c}=c$, Table (\ref{pure.branch}) presents properties
on the valuations $\nu_{r}$ depending also on the type of what is
called in that case a \emph{pure} mixed branch.
\begin{center}
\begin{table}
\begin{tabular}{ccc}
 & $\nu_{r}$ & $\nu_{m_{c}^{+}}$ \tabularnewline[\doublerulesep]
\hline 
\noalign{\vskip\doublerulesep}
\includegraphics[width=0.1\paperwidth]{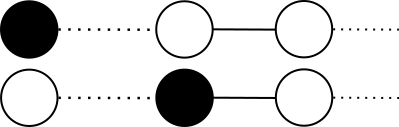} & $\textup{odd}$ & $\textup{even}$\tabularnewline[\doublerulesep]
\hline 
\noalign{\vskip\doublerulesep}
\includegraphics[width=0.1\paperwidth]{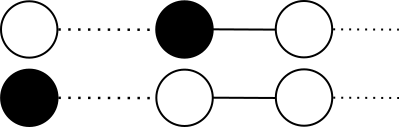} & $\textup{even}$ & $\textup{odd}$\tabularnewline[\doublerulesep]
\hline 
\noalign{\vskip\doublerulesep}
\includegraphics[width=0.1\paperwidth]{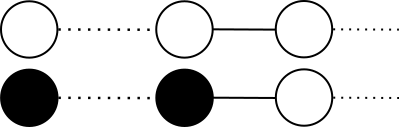} & $\textup{even}$ & $\textup{even}$\tabularnewline[\doublerulesep]
\hline 
\noalign{\vskip\doublerulesep}
\includegraphics[width=0.1\paperwidth]{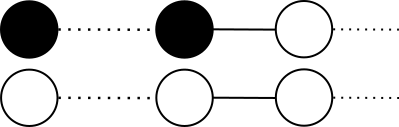} & $\textup{odd}$ & $\textup{odd}$\tabularnewline[\doublerulesep]
\end{tabular}$\qquad$%
\begin{tabular}{ccc}
 & $\nu_{r}$ & $\nu_{m_{c}^{+}}$ \tabularnewline[\doublerulesep]
\hline 
\noalign{\vskip\doublerulesep}
\includegraphics[width=0.1\paperwidth]{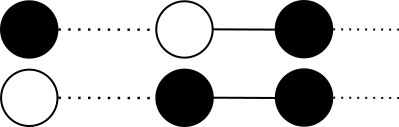} & $\textup{odd}$ & $\textup{odd}$\tabularnewline[\doublerulesep]
\hline 
\noalign{\vskip\doublerulesep}
\includegraphics[width=0.1\paperwidth]{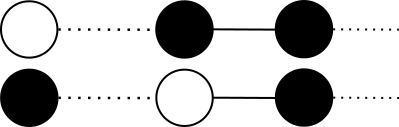} & $\textup{even}$ & $\textup{even}$\tabularnewline[\doublerulesep]
\hline 
\noalign{\vskip\doublerulesep}
\includegraphics[width=0.1\paperwidth]{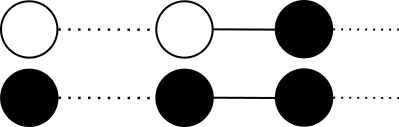} & $\textup{even}$ & $\textup{odd}$\tabularnewline[\doublerulesep]
\hline 
\noalign{\vskip\doublerulesep}
\includegraphics[width=0.1\paperwidth]{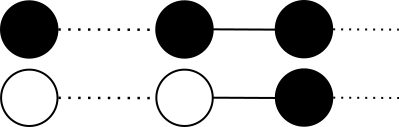} & $\textup{odd}$ & $\textup{even}$\tabularnewline[\doublerulesep]
\end{tabular}

\caption{\label{Parit=0000E9 de nu}The valuations $\nu_{r}$ and $\nu_{m_{c}^{+}}$
along a mixed branch.}
\end{table}
\par\end{center}

\begin{center}
\begin{table}
\begin{centering}
\begin{tabular}{cc}
$\left|\AA_{c}\right|>1$ & $\nu_{r}$\tabularnewline[\doublerulesep]
\hline 
\noalign{\vskip\doublerulesep}
\includegraphics[width=0.05\paperwidth]{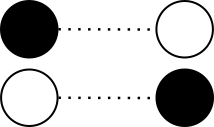} & $\textup{odd}$\tabularnewline[\doublerulesep]
\hline 
\noalign{\vskip\doublerulesep}
\includegraphics[width=0.05\paperwidth]{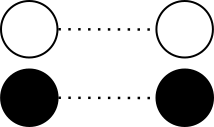} & $\textup{even}$\tabularnewline[\doublerulesep]
\hline 
\noalign{\vskip\doublerulesep}
\includegraphics[width=0.05\paperwidth]{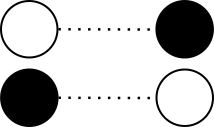} & $\textup{impossible}$\tabularnewline[\doublerulesep]
\hline 
\noalign{\vskip\doublerulesep}
\includegraphics[width=0.05\paperwidth]{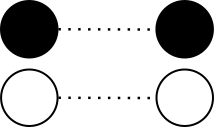} & $\textup{impossible}$\tabularnewline[\doublerulesep]
\end{tabular}$\qquad$%
\begin{tabular}{cc}
$\left|\AA_{c}\right|=1$ & $\nu_{r}$\tabularnewline[\doublerulesep]
\hline 
\noalign{\vskip\doublerulesep}
\includegraphics[width=0.0145\paperwidth]{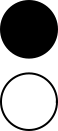} & $\textup{odd}$\tabularnewline[\doublerulesep]
\hline 
\noalign{\vskip\doublerulesep}
\includegraphics[width=0.0145\paperwidth]{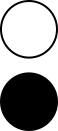} & $\textup{even}$\tabularnewline[\doublerulesep]
\end{tabular}
\par\end{centering}
\caption{\label{pure.branch}The valuation $\nu_{r}$ along a \emph{pure} mixed
branch.}
\end{table}
\par\end{center}
\item \label{component.dicritique}Finally, for each connected component
$\mathbb{K}$ of the sub-graph $\mathbb{A}\setminus\left\{ \left.c\in\AA\right|\Delta_{c}^{n}=0\right\} ,$
there exists $c\in\mathbb{K}$ with $\epsilon_{c}^{n}>0.$
\end{enumerate}
\end{thm}

Property (\ref{Existence.Unicite.Saito}) of Theorem \ref{Theoreme.Un}
has already been proved in \cite{moduligenz} for the particular case
of a tree $\AA$ for which $\rho_{c}=1$ for any $c\in\AA.$
\begin{example}
Let us consider the \emph{cusp tree, }that is the tree in Figure \ref{fig : cusp}
numbered by $n_{1}=n_{2}=0$ and $n_{3}=1.$ The order is defined
by the relations $1\leq2$ and $1\leq3$.

The proximity matrix is 
\[
\mathbb{P}=\left(\begin{array}{ccc}
1 & -1 & -1\\
0 & 1 & -1\\
0 & 0 & 1
\end{array}\right)
\]
Figure \ref{fig : cusp} presents also the associated Saito dicriticity
represented by the coloring. 

\begin{figure}
\begin{centering}
\includegraphics[scale=0.4]{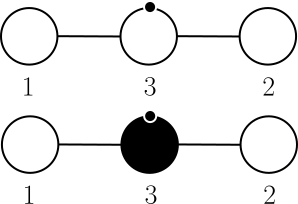}
\par\end{centering}
\caption{\label{fig : cusp}The \emph{cusp} tree and its Saito dicriticity}
\end{figure}

Table (\ref{tab:Dicriticities-and-configurations}) shows the various
configurations obtained from the $8=2^{3}$ possible dicriticities
on $\mathbb{\AA}.$ For each configuration, we highlight by the notation
$\left\langle \cdot\right\rangle $ a part of the configuration that
violates one of the admissibility conditions. At the end, the unique
and thus Saito dicriticity that satisfies all the three admissibility
conditions is $\left(1,1,0\right)$ for which 
\begin{align*}
\epsilon_{1}=1\geq0 & ,\ \Delta_{1}=1\\
\epsilon_{2}=1\geq0 & ,\ \Delta_{2}=1\\
\epsilon_{3}=0\geq2-\Delta_{1}-\Delta_{2} & ,\ \Delta_{3}=0.
\end{align*}
Notice that, $1$ and $2$ are two components of $\AA\setminus\left\{ 3\right\} $
for which $\epsilon_{1}>0$ and $\epsilon_{2}>0$, as predicted by
the property (\ref{component.dicritique}) of Theorem \ref{Theoreme.Un}.

\begin{table}
\begin{centering}
\begin{tabular}{cc}
$\Delta=\left(\Delta_{1},\Delta_{2},\Delta_{3}\right)$ & $\epsilon=\left(\epsilon_{1},\epsilon_{2,}\epsilon_{3}\right)$\tabularnewline
\hline 
$0,0,0$ & $\left\langle -1\right\rangle ,0,1$\tabularnewline
\noalign{\vskip\doublerulesep}
\hline 
$0,0,1$ & $\left\langle -1\right\rangle ,0,1$\tabularnewline
\noalign{\vskip\doublerulesep}
\hline 
$0,1,0$ & $\left\langle 0\right\rangle ,1,1$\tabularnewline
\noalign{\vskip\doublerulesep}
\hline 
$0,1,1$ & $\left\langle -1\right\rangle ,0,1$\tabularnewline
\noalign{\vskip\doublerulesep}
\hline 
$1,0,0$ & $2,\left\langle 0\right\rangle ,0$\tabularnewline
\noalign{\vskip\doublerulesep}
\hline 
$1,0,1$ & $2,\left\langle 0\right\rangle ,1$\tabularnewline
\noalign{\vskip\doublerulesep}
\hline 
$1,1,0$ & $1,1,0$\tabularnewline
\noalign{\vskip\doublerulesep}
\hline 
$1,1,1$ & $1,1,\left\langle 0\right\rangle $\tabularnewline
\noalign{\vskip\doublerulesep}
\end{tabular}
\par\end{centering}
\caption{\label{tab:Dicriticities-and-configurations}Dicriticities and configurations
for the cusp tree numbered by $\left(0,0,1\right).$}
\end{table}
\end{example}

\begin{example}
Suppose that $\AA$ is a tree reduced to two vertices. Its proximity
matrix is
\[
\text{\ensuremath{\mathbb{P}}}=\left(\begin{array}{cc}
1 & -1\\
0 & 1
\end{array}\right).
\]
Figure \ref{fig:Unique-admissible-choice} presents the Saito dicriticity
$\Delta=\left(\star,\star\right)\in\left\{ 0,1\right\} ^{2}$ depending
on the values $n_{1}$ and $n_{2}$.
\begin{figure}
\begin{centering}
\includegraphics[scale=0.5]{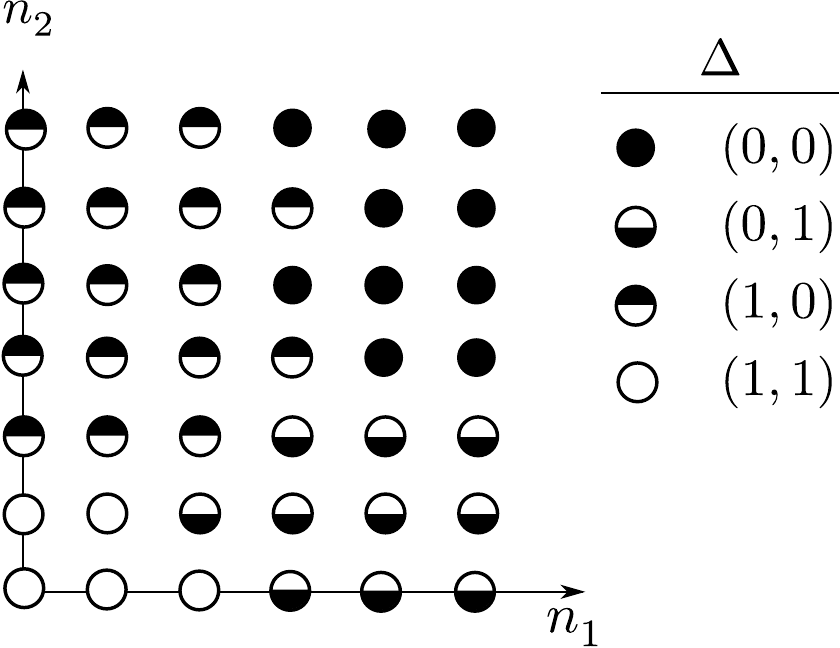}
\par\end{centering}
\caption{\label{fig:Unique-admissible-choice}Unique admissible choice of $\Delta=\left(\Delta_{1},\Delta_{2}\right).$}
\end{figure}
\end{example}

Let $c\in\AA.$ Along the row associated to $c$ in the matrix $\mathbb{P}$,
any occurrence of a coefficient $-1$ corresponds to a vertex that
belongs to the access tree from $r$ to $v$ for some $v$ in the
neighborhood $\mathfrak{v}_{r}.$ This remark leads to the following
expression of $\epsilon_{c}$ that is going to be used extensively
in the sequel, 
\begin{align}
\epsilon_{c} & =\frac{n_{c}}{2}-\square_{c}+\sum_{v\in\voisin c}\sum_{s\in\Aacces cv{}\setminus\left\{ c\right\} }\square_{s}.\nonumber \\
 & =\frac{n_{c}}{2}-\frac{\delta_{c}}{2}+\pair{\Delta_{c}}{\frac{1}{2}}{\nu_{c}^{n}-\delta_{c}}+\sum_{v\in\voisin c}\sum_{s\in\Aacces cv{}\setminus\left\{ c\right\} }\frac{\delta_{s}}{2}-\pair{\Delta_{s}}{\frac{1}{2}}{\nu_{s}^{n}-\delta_{s}}.\label{Formule fondamentale}
\end{align}

\begin{proof}[Proof of Theorem \ref{Theoreme.Un}.]
The proof is, as a whole, an induction on the number of vertices
in $\AA.$ 

Suppose that $\left|\AA\right|=1$. The proximity matrix is $\mathbb{P}=\left(1\right)$
and the numbering $n=\left(n_{r}\right).$ In view of (\ref{Formule fondamentale}),
we get 
\begin{align*}
\epsilon_{r} & =\frac{n_{r}}{2}-\square_{r}=\frac{n_{r}}{2}-\left(\frac{\delta_{r}}{2}-\pair{\Delta_{r}}{\frac{1}{2}}{\nu_{r}^{n}-\delta_{r}}\right)\\
 & =\frac{n_{r}}{2}+\pair{\Delta_{r}}{\frac{1}{2}}{n_{r}}
\end{align*}
since $\delta_{r}=0$ and $\nu_{r}^{n}=n_{r}.$ Suppose $n_{r}=0,1$
or $2.$ Then it can be seen that $\Delta_{r}=0$ is not admissible,
since it would impose that $\epsilon_{r}\geq2,$ which is not true.
However, if $\Delta_{r}=1,$ for $n_{r}=0,1$ or $2,$ we have respectively
$\epsilon_{r}=1,\ 1$ and $2$ that always satisfies $\epsilon_{r}\geq n_{r}.$
Thus for $n_{r}=0,1$ or $2$, the Saito dicriticity is defined by
$\Delta_{r}^{n_{r}}=1$. To the contrary, if $n_{r}\geq3$ then 
\[
\epsilon_{r}\leq\frac{n_{r}}{2}+1<n_{r}
\]
thus $\Delta_{r}=1$ is not an admissible dicriticity. However, $\Delta_{r}^{n_{r}}=0$
is admissible since 
\[
\epsilon_{r}=\frac{n_{r}}{2}+\pair 0{\frac{1}{2}}{n_{r}}\geq2,
\]
which concludes the proof of the property (\ref{Existence.Unicite.Saito})
for $\left|\AA\right|=1.$

Now consider property (\ref{propriete.theta1}) and the invariant
$\Tun r{\AA}{n_{r}}$. By specifying the definition, we obtain
\begin{align*}
\Tun r{\AA}n & =\square_{r}^{n}+\square_{r}^{r\cdot n}\\
 & =\frac{\delta_{r}^{n}}{2}-\pair{\Delta_{r}^{n}}{\frac{1}{2}}{\nu_{r}^{n}-\delta_{r}^{n}}+\frac{\delta_{r}^{r\cdot n}}{2}-\pair{\Delta_{r}^{r\cdot n}}{\frac{1}{2}}{\nu_{r}^{r\cdot n}-\delta_{r}^{n}}
\end{align*}

Notice that $\delta_{r}^{n}=\delta_{r}^{r\cdot n}=0$, $\nu_{r}^{n}=n_{r}$
and $\nu_{r}^{r\cdot n}=n_{r}+1$. Consequently, this gives
\[
\Tun r{\AA}n=-\pair{\Delta_{r}^{n}}{\frac{1}{2}}{n_{r}}-\pair{\Delta_{r}^{r\cdot n}}{\frac{1}{2}}{n_{r}+1}=-\frac{1}{2}-\pair{\Delta_{r}^{n}}{\Delta_{r}^{r\cdot n}}{n_{r}}
\]
If $n_{r}=0$ or $1$ then $\Delta_{r}^{n}=\Delta_{r}^{r\cdot n}=1$
and thus the invariant is written 
\[
\Tun r{\AA}n=-\frac{3}{2}=-\frac{\left|\mathbb{A}_{r}\right|}{2}-\Delta_{r}^{n}.
\]
If $n_{r}\geq3$ then $\Delta_{r}^{n}=\Delta_{r}^{r\cdot n}=0$ which
induces 
\[
\Tun r{\AA}n=-\frac{1}{2}=-\frac{\left|\mathbb{A}_{r}\right|}{2}-\Delta_{r}^{n}.
\]
Finally, if $n_{r}=2$ then $\Delta_{r}^{n}=1$ and $\Delta_{r}^{r\cdot n}=0$
and hence $\pair{\Delta_{r}^{n}}{\Delta_{r}^{r\cdot n}}{n_{r}}=1$.
Therefore, it follows
\[
\Tun r{\AA}n=-\frac{3}{2}=-\frac{\left|\mathbb{A}_{r}\right|}{2}-\Delta_{r}^{n}.
\]
Consequently, formula (\ref{Valeur de Theta1}) is true for $\left|\AA\right|=1$.
Now, the invariant $\Teu rr{\mathbb{A}}n$ is written
\[
\Teu rr{\AA}n=\square_{r}^{n}+\square_{r}^{r\cdot r\cdot n}=-\pair{\Delta_{r}^{n}}{\frac{1}{2}}{n_{r}}-\pair{\Delta_{r}^{r\cdot r\cdot n}}{\frac{1}{2}}{n_{r}+2}=-\pair{\Delta_{r}^{n}+\Delta_{r}^{r\cdot r\cdot n}}1{n_{r}}.
\]
If $n_{r}=0$ then it reduces to 
\[
-\pair{\Delta_{r}^{n}+\Delta_{r}^{r\cdot r\cdot n}}1{n_{r}}=-2=-\frac{1}{2}-\frac{\left|\text{\ensuremath{\Aacces{}r{}}}\right|}{2}-\Delta_{r}^{n}.
\]

If $n_{r}=1,2$ then we obtain
\[
-\pair{\Delta_{r}^{n}+\Delta_{r}^{r\cdot r\cdot n}}1{n_{r}}=-1=\frac{1}{2}-\frac{\left|\text{\ensuremath{\Aacces{}r{}}}\right|}{2}-\Delta_{r}^{n}.
\]
Finally, if $n_{r}\geq3,$ then we get
\[
-\pair{\Delta_{r}^{n}+\Delta_{r}^{r\cdot r\cdot n}}1{n_{r}}=-\pair 01{n_{r}}=\pm\frac{1}{2}-\frac{\left|\text{\ensuremath{\Aacces{}r{}}}\right|}{2}-\Delta_{r}^{n},
\]
thus formula (\ref{Valeur de Theta2}) holds. We continue in this
fashion obtaining the invariant $\Tnu rr{\AA}n$,
\[
\Tnu rr{\AA}n=\square_{r}^{r\cdot n}+\square_{r}^{r\cdot n}=-\pair{2\Delta_{r}^{r\cdot n}}1{n_{r}+1}.
\]
 If $n_{r}=0,1,2$ then $\pm\frac{1}{2}-\frac{\left|\text{\ensuremath{\Aacces{}r{}}}\right|}{2}-\Delta_{r}^{n}=-2\textup{ or }-1$.
This implies (\ref{Valeur de Theta11}). If $n_{r}\geq3,$ then we
get
\[
-\pair{2\Delta_{r}^{r\cdot n}}1{n_{r}+1}=-\pair 01{n_{r}+1}=\pm\frac{1}{2}-\frac{\left|\text{\ensuremath{\Aacces{}r{}}}\right|}{2}-\Delta_{r}^{n},
\]
which completes the proof of formula (\ref{Valeur de Theta11}). For
$\left|\AA\right|=1,$ there are only pure mixed branches of length
one. It is enough to refer to the computations above to obtain the
following correspondance : for $\AA=\left\{ r\right\} $, we get 
\begin{center}
\begin{tabular}{ccccc}
$n=\left(n_{r}\right)$ & $0$ & $1$ & $2$ & $\geq3$\tabularnewline
\hline 
$\Delta_{r}^{n}$ & $1$ & $1$ & $1$ & $0$\tabularnewline
\end{tabular},
\par\end{center}

which ensures the properties in Table \ref{Mixed.branch}. To conclude
the case $\left|\AA\right|=1,$we observe that property (\ref{component.dicritique})
follows from the computations of $\epsilon_{r}^{n}$ for $n_{r}=0,1$
and $2$. 

Now, we are going to prove inductively property (\ref{Existence.Unicite.Saito})
from properties (\ref{Existence.Unicite.Saito}) and (\ref{propriete.theta1}).\textcolor{black}{{}
Let us consider the $\left|\voisin r\right|$ graphs obtained as the
connected components of $\AA\setminus\left\{ r\right\} $. We index
these graphs by $\voisin r$ itself by denoting each connected component
$\AA^{c}$ for $c\in\voisin r$. Each tree $\AA^{c}$ inherits an
order from the one of $\AA$. Let us consider two different numberings
of each component $\AA^{c},$ $c\in\voisin r$. In the sequel, we
refer to these two different numbered trees by the notation $\AA^{\star,\nu}$
with $\star=0$ or 1. }
\begin{itemize}
\item $\star=0,\ $$\mathbb{A}^{0,c}=\AA^{c}$ numbered by the integer $n^{0}=\left(n_{s}^{0}\right)_{s\in\mathbb{A}^{0,c}}$
with $n_{s}^{0}=n_{s}$ for $s\in\mathbb{A}^{0,c}.$
\item $\star=1,~$$\mathbb{A}^{1,c}=\AA^{c}$ but numbered by the integer
$n^{1}=\left(n_{s}^{1}\right)_{s\in\mathbb{A}^{1,c}}$ with $n_{s}^{1}=n_{s}$
for $c\neq s\in\mathbb{A}^{1,c}$, and $n_{c}^{1}=n_{c}+1.$
\end{itemize}
Note that by construction the tree $\AA$ is obtained by gluing the
family of trees $\left(\AA^{\star,c}\right)_{c\in\voisin r}$ with
$\star=0$ or $1$ along the vertex $r$ adding an edge between each
vertex $c$ and the root $r$. Each vertex $s$ belongs exactly to
one of the trees $\AA^{\star,c}.$ Applying property (\ref{Existence.Unicite.Saito})
to each numbered tree $\mathbb{A}^{\star,c}$, we obtain a family
of dicriticities $\Delta^{\star,c}$, that consists in the family
of unique Saito dicriticities of the numbered trees $\mathbb{A}^{\star,c}.$
As a result, we can define two new distinct dicriticities on the whole
tree $\AA$ induced by the $\Delta^{\star,c},$ $c\in\voisin r$ the
following way :
\begin{itemize}
\item $\Delta^{1}$, $\Delta_{r}^{1}=1$ and for any $s\neq r$ $\Delta_{s}^{1}=\Delta_{s}^{1,c}$
if $s\in\AA^{1,c}$. 
\item $\Delta^{0}$, $\Delta_{r}^{0}=0$ and for any $s\neq r$ $\Delta_{s}^{0}=\Delta_{s}^{0,c}$
if $s\in\AA^{0,c}$.
\end{itemize}
We claim that both dicriticities $\Delta^{0}$ or $\Delta^{1}$ satisfy
the admissibility conditions for the vertices $s\in\AA\setminus\left\{ r\right\} .$
Indeed, let us denote $\star^{\star,c}$ the combinatorial datas resulting
from property (\ref{Existence.Unicite.Saito}) applied to each numbered
tree $\AA^{\star,c}$. We also denote simply by $\star^{0\textup{ or }1}$
the combinatorial datas associated respectively to the dicriticities
$\Delta^{0}$ or $\Delta^{1}.$

Let us focus first on the dicriticity $\Delta^{1}$. For $c\in\voisin r$
and $s\in\AA^{1,c}$, we get
\begin{equation}
\begin{array}{rl}
\textup{ if }s\notin\Aaccess{}c{1,c}\setminus\left\{ r\right\} , & \nu_{s}^{1,c}=\nu_{s}^{1},\ \delta_{s}^{1,c}=\delta_{s}^{1}\\
\textup{ if }s\in\Aaccess{}c{1,c}\setminus\left\{ r\right\} , & \nu_{s}^{1,c}=\nu_{s}^{1}+1,\ \delta_{s}^{1,c}=\delta_{s}^{1}-1
\end{array}\label{relation.induction}
\end{equation}
Note that in any case, $\nu_{s}^{1,c}-\delta_{s}^{1,c}$ and $\nu_{s}^{1}-\delta_{s}^{1}$
have the same parity. If $s\notin\Aaccess{}c{1,c}\setminus\left\{ r\right\} $,
relations (\ref{relation.induction}) combined with the construction
of $\Delta^{1}$ ensures that $\epsilon_{s}^{1,c}=\epsilon_{s}^{1}.$

If $s\in\Aaccess{}c{1,c}\setminus\left\{ r\right\} $ and $s\neq c$
then we obtain 
\begin{align*}
\epsilon_{s}^{1,c} & =\frac{n_{s}}{2}-\square_{s}^{1,c}+\sum_{v\in\voisin s}\sum_{u\in\Aacces sv{}\setminus\left\{ s\right\} }\square_{u}^{1,c}\\
 & =\frac{n_{s}}{2}-\square_{s}^{1,c}+\square_{c}^{1,c}+\sum_{v\in\voisin s}\sum_{u\neq c\in\Aacces sv{}\setminus\left\{ s\right\} }\square_{u}^{1,c}\\
 & =\frac{n_{s}}{2}-\frac{-1+\delta_{s}^{1}}{2}+\pair{\Delta_{s}^{1}}{\frac{1}{2}}{\nu_{s}^{1,c}-\delta_{s}^{1,c}}+\frac{-1+\delta_{c}^{1}}{2}-\pair{\Delta_{c}^{1}}{\frac{1}{2}}{\nu_{c}^{1,c}-\delta_{c}^{1,c}}+\sum_{v\in\voisin s}\sum_{u\neq c\in\Aacces sv{}\setminus\left\{ s\right\} }\square_{u}^{1,c}\\
 & =\frac{n_{s}}{2}-\frac{\delta_{s}^{1}}{2}+\pair{\Delta_{s}^{1}}{\frac{1}{2}}{\nu_{s}^{1}-\delta_{s}^{1}}+\frac{\delta_{c}^{1}}{2}-\pair{\Delta_{c}^{1}}{\frac{1}{2}}{\nu_{c}^{1}-\delta_{c}^{1}}+\sum_{v\in\voisin s}\sum_{u\neq c\in\Aacces sv{}\setminus\left\{ s\right\} }\square_{u}^{1}\\
 & =\epsilon_{s}^{1}.
\end{align*}
If $s=c$, it follows from the numbering of $\AA^{1,c}$ that
\begin{align*}
\epsilon_{c}^{1,c} & =\frac{n_{c}+1}{2}-\square_{c}^{1,c}+\sum_{v\in\voisin c}\sum_{u\in\Aacces sv{}\setminus\left\{ s\right\} }\square_{u}^{1,c}\\
 & =\frac{n_{c}+1}{2}-\frac{\delta_{c}^{1,c}}{2}+\pair{\Delta_{c}^{1,c}}{\frac{1}{2}}{\nu_{c}^{1,c}-\delta_{c}^{1,c}}+\sum_{v\in\voisin c}\sum_{u\in\Aacces cv{}\setminus\left\{ c\right\} }\square_{u}^{1,c}\\
 & =\frac{n_{c}+1}{2}-\frac{-1+\delta_{c}^{1}}{2}+\pair{\Delta_{c}^{1}}{\frac{1}{2}}{\nu_{c}^{1}-\delta_{c}^{1}}+\sum_{v\in\voisin c}\sum_{u\in\Aacces cv{}\setminus\left\{ c\right\} }\square_{u}^{1}\\
= & \epsilon_{c}^{1}+1.
\end{align*}
Since the configuration $\left(\epsilon^{1,c}\right)_{s}$ is admissible
for $\AA^{1,c}$, $\epsilon_{s}^{1}$ satifies the admissibility conditions
for $s\neq c$. For $s=c,$ if $\Delta_{c}^{1}=1$ then we get the
following inequality
\[
\epsilon_{c}^{1}=\epsilon_{c}^{1,c}-1\geq n_{c}^{1}+1-1\geq n_{c}^{1},
\]
and if $\Delta_{c}^{1}=0$ then the relation becomes
\begin{align*}
\epsilon_{c}^{1} & =\epsilon_{c}^{1,c}-1\geq\left(2-\sum_{s\in\voisin c\setminus\left\{ c\right\} }\Delta_{s}^{1,c}\right)-1\\
 & \geq\left(2-\sum_{s\in\voisin c\setminus\left\{ c\right\} }\Delta_{s}^{1}\right)-\Delta_{r}^{1}\\
 & \geq2-\sum_{s\in\voisin c}\Delta_{s}^{1}.
\end{align*}
Thus, in any case, the configuration $\epsilon^{1}$ is admissible
for $s\neq r$. Using much the same computations, we can prove that
$\epsilon^{0}$ is also admissible for $s\neq r.$ 

However, we are going to prove that exactly one dicriticity among
$\Delta^{0}$ and $\Delta^{1}$ satisfies the admissibility condition
for $s=r.$ Indeed, we have
\begin{align*}
\epsilon_{r}^{0}+\epsilon_{r}^{1} & =\frac{n_{r}}{2}-\square_{r}^{0}+\left(\sum_{v\in\voisin r}\sum_{s\in\Aacces{}v{}\setminus\left\{ r\right\} }\square_{s}^{0}\right)+\frac{n_{r}}{2}-\square_{r}^{1}+\left(\sum_{v\in\voisin c}\sum_{s\in\Aacces cv{}\setminus\left\{ c\right\} }\square_{s}^{1}\right)\\
 & =n_{r}+\pair 0{\frac{1}{2}}{\nu_{r}}+\pair 1{\frac{1}{2}}{\nu_{r}}+\sum_{v\in\voisin r}\sum_{s\in\Aacces{}v{}\setminus\left\{ r\right\} }\square_{s}^{0}+\square_{s}^{1}\\
 & =n_{r}+1+\sum_{v\in\voisin r}\sum_{s\in\Aacces{}v{}\setminus\left\{ r\right\} }\square_{s}^{0}+\square_{s}^{1}.
\end{align*}

Now, if $v\in\voisin r$ and $s\in\Aacces rv{}\setminus\left\{ r\right\} $
one has 
\[
\square_{s}^{0}+\square_{s}^{1}=\frac{\delta_{s}^{0}}{2}-\pair{\Delta_{s}^{0}}{\frac{1}{2}}{\nu_{s}^{0}-\delta_{s}^{0}}+\frac{\delta_{s}^{1}}{2}-\pair{\Delta_{s}^{1}}{\frac{1}{2}}{\nu_{s}^{1}-\delta_{s}^{1}}
\]
The construction of $\Delta^{\star}$ and the relations (\ref{relation.induction})
force
\[
\square_{s}^{0}+\square_{s}^{1}=\frac{1}{2}+\square_{s}^{0,s}+\square_{s}^{1,s}
\]
which leads to 
\begin{align*}
\epsilon_{r}^{0}+\epsilon_{r}^{1} & =n_{r}+1+\sum_{v\in\voisin r}\sum_{s\in\Aacces{}v{}\setminus\left\{ r\right\} }\frac{1}{2}+\square_{s}^{0,s}+\square_{s}^{1,s}\\
 & =n_{r}+1+\sum_{v\in\voisin r}\frac{\left|\Aacces{}v{}\setminus\left\{ r\right\} \right|}{2}+\Tun v{\mathbb{A}^{0,v}}n
\end{align*}
Notice that in the tree $\AA^{0,\nu}$ the vertex $\nu$ is of multiplicity
$1$. Property (\ref{propriete.theta1}) gives the relation 
\[
\Tun{\nu}{\mathbb{A}^{0,\nu}}n=-\frac{\left|\Aacces{}v{}\setminus\left\{ r\right\} \right|}{2}-\Delta_{v}^{0}.
\]
and the sum $\epsilon_{r}^{0}+\epsilon_{r}^{1}$ reduces to
\[
\epsilon_{r}^{0}+\epsilon_{r}^{1}=n_{r}+1-\sum_{v\in\voisin r}\Delta_{v}^{0}.
\]
Finally, the above equality ensures that one of the two conditions
\[
\epsilon_{r}^{1}\geq n_{r}\qquad\textup{ or }\qquad\epsilon_{r}^{0}\geq2-\sum_{v\in\voisin r}\Delta_{v}^{0}
\]
holds but not both. As a consequence, either $\Delta^{1}$ or $\Delta^{0}$
is admissible for $s=r,$ but not both. That concludes the proof of
property (\ref{Existence.Unicite.Saito}) for the tree $\mathbb{A}.$

Now, we are going to prove property (\ref{propriete.theta1}) inductively
from properties (\ref{Existence.Unicite.Saito}), (\ref{propriete.theta1})
and (\ref{Mixed.branch}). Suppose first that the Saito dicricities
respectively associated to $n$ and $c\cdot n$ start with \includegraphics[viewport=0bp 29.92366bp 118bp 98bp,scale=0.15]{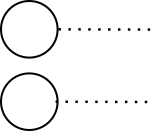}
, then the invariant $\Tun c{\AA}n$ is written 
\begin{align*}
\Tun c{\AA}n & =\sum_{s\in\Aacces{}c{}}\square_{s}^{n}+\square_{s}^{c\cdot n}\\
 & =\square_{r}^{n}+\square_{r}^{c\cdot n}+\sum_{s\in\Aacces{}c{},\ s\neq r}\square_{s}^{n}+\square_{s}^{c\cdot n}\\
 & =-\pair 1{\frac{1}{2}}{\nu_{r}^{n}}-\pair 1{\frac{1}{2}}{\nu_{r}^{n}+1}+\underbrace{\frac{\delta_{r^{+}}^{n}}{2}+\frac{\delta_{r^{+}}^{c\cdot n}}{2}}_{=1}-\left(\frac{\delta_{r^{+}}^{n}}{2}+\frac{\delta_{r^{+}}^{c\cdot n}}{2}\right)+\sum_{s\in\Aacces{}c{},\ s\neq r}\square_{s}^{n}+\square_{s}^{c\cdot n}\\
 & =-\frac{1}{2}-\left(\frac{\delta_{r^{+}}^{n}}{2}+\frac{\delta_{r^{+}}^{c\cdot n}}{2}\right)+\sum_{s\in\Aacces{}c{},\ s\neq r}\square_{s}^{n}+\square_{s}^{c\cdot n}
\end{align*}
where $r^{+}$ is the successor of $r$ in the branch $\AA_{c}.$
Consider the tree $\AA^{r^{+}}$connected composent of $r^{+}$ in
$\AA\setminus\left\{ r\right\} $. The inherited order of $\AA^{r{{}^+}}$
makes of $r^{+}$ its root. Let $s_{0}$ be the vertex in the neighbobrhood
$\voisin r$ of $r$ in $\AA$ such that $s_{0}\geq r^{+}.$ Note
that from the unicity statement of property (\ref{Existence.Unicite.Saito})
inductively applied to $\AA^{r^{+}}$, the dicriticity of $\AA^{r^{+}}$
inherited from the Saito dicriticity of $\AA$ numbered by $n$ is
the Saito dicriticity of $\AA^{r^{+}}$ numbered by the numbering
inherited from the $s_{0}\cdot n$. Applying inductively property
(\ref{propriete.theta1}) to the tree $\AA^{r^{+}}$, we get 
\begin{align*}
-\left(\frac{\delta_{r^{+}}^{n}}{2}+\frac{\delta_{r^{+}}^{c\cdot n}}{2}\right)+\sum_{s\in\Aacces{}c{},\ s\neq r}\square_{s}^{n}+\square_{s}^{c\cdot n} & =\sum_{s\in\Aacces{}c{r^{+}},}\square_{s}^{s_{0}\cdot n}+\square_{s}^{c\cdot s_{0}\cdot n}\\
 & =\Tun c{\AA^{r^{+}}}{s_{0}\cdot n}\\
 & =-\frac{\left|\Aacces{}c{}\right|-1}{2}-\Delta_{c}^{s_{0}\cdot n}.
\end{align*}
Combining the two above relations, we are lead to 
\[
\Tun c{\AA}n=-\frac{1}{2}-\frac{\left|\Aacces{}c{}\right|-1}{2}-\Delta_{c}^{s_{0}\cdot n}=-\frac{\left|\Aacces{}c{}\right|}{2}-\Delta_{c}^{n},
\]
which is property (\ref{propriete.theta1}). Now, if the Saito dicricities
associated to $n$ and $c\cdot n$ start with \includegraphics[viewport=0bp 29.92366bp 113bp 98bp,scale=0.15]{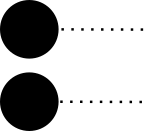}
, then the invariant $\Tun c{\AA}n$ becomes
\begin{align*}
\Tun c{\AA}n & =\sum_{s\in\Aacces{}c{}}\square_{s}^{n}+\square_{s}^{c\cdot n}\\
 & =\square_{r}^{n}+\square_{r}^{c\cdot n}+\sum_{s\in\Aacces{}c{},\ s\neq r}\square_{s}^{n}+\square_{s}^{c\cdot n}\\
 & =-\pair 0{\frac{1}{2}}{\nu_{r}^{n}}-\pair 0{\frac{1}{2}}{\nu_{r}^{n}+1}+\underbrace{\frac{\delta_{r^{+}}^{n}}{2}+\frac{\delta_{r^{+}}^{c\cdot n}}{2}}_{=0}-\left(\frac{\delta_{r^{+}}^{n}}{2}+\frac{\delta_{r^{+}}^{c\cdot n}}{2}\right)+\sum_{s\in\Aacces{}c{},\ s\neq r}\square_{s}^{n}+\square_{s}^{c\cdot n}\\
 & =-\frac{1}{2}-\left(\frac{\delta_{r^{+}}^{n}}{2}+\frac{\delta_{r^{+}}^{c\cdot n}}{2}\right)+\sum_{s\in\Aacces{}c{},\ s\neq r}\square_{s}^{n}+\square_{s}^{c\cdot n}
\end{align*}
As above, applying inductively property (\ref{propriete.theta1})
yields 
\begin{align*}
-\left(\frac{\delta_{r^{+}}^{n}}{2}+\frac{\delta_{r^{+}}^{c\cdot n}}{2}\right)+\sum_{s\in\Aacces{}c{},\ s\neq r}\square_{s}^{n}+\square_{s}^{c\cdot n} & =\sum_{s\in\Aacces{}c{r^{+}},}\square_{s}^{n}+\square_{s}^{c\cdot n}\\
 & =\Tun c{\AA^{r^{+}}}n.\\
 & =-\frac{\left|\Aacces{}c{}\right|-1}{2}-\Delta_{c}^{n}.
\end{align*}
As before, the two above relations lead to 
\[
\Tun c{\AA}n=-\frac{1}{2}-\frac{\left|\Aacces{}c{}\right|-1}{2}-\Delta_{c}^{n}=-\Delta_{c}^{n}-\frac{\left|\Aacces{}c{}\right|}{2},
\]
which is the desired property. We now turn to the case in which the
Saito dicricities associated to $n$ and $c\cdot n$ start with \includegraphics[viewport=0bp 29.92366bp 116bp 98bp,scale=0.15]{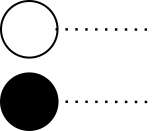}
. Hence, we are in the presence of a mixed or pure mixed branch. Suppose
first that $\left|\AA_{c}\right|=1$. In that case, the branch is
pure and the invariant $\Tun c{\AA}n$ reduces to
\[
\Tun c{\AA}n=-\pair{\Delta_{r}^{n}}{\frac{1}{2}}{\nu_{r}^{n}}-\pair{\Delta_{r}^{r\cdot n}}{\frac{1}{2}}{\nu_{r}^{r\cdot n}}=-\pair 1{\frac{1}{2}}{\nu_{r}^{n}}-\pair 0{\frac{1}{2}}{\nu_{r}^{n}+1}.
\]
From Table \ref{Mixed.branch}, we get 
\[
\Tun c{\AA}n=-1-\frac{1}{2}=-\frac{\left|\AA_{c}\right|}{2}-\Delta_{c}^{n}.
\]

Suppose now that $\left|\AA_{c}\right|\geq2$. In that case, along
the mixed branch, the nature of the square index allows us to simplify
the expression of the invariant $\Tun c{\AA}n$. Suppose that $s$
and $s^{\prime}$ are consecutive vertices $s\leq s^{\prime}$ in
$\Aacces{}c{}$ with 
\begin{equation}
\Delta_{\star}^{n}+\Delta_{\star}^{c\cdot n}=1,\ \star=s,\ s^{\prime}.\label{eq:1}
\end{equation}
Then, evaluating the square indeces at $s^{\prime}$ yields 
\[
\square_{s^{\prime}}^{n}+\square_{s^{\prime}}^{c\cdot n}=\frac{\delta_{s^{\prime}}^{n}}{2}+\frac{\delta_{s^{\prime}}^{c\cdot n}}{2}-\pair{\Delta_{s^{\prime}}^{n}}{\frac{1}{2}}{\nu_{s^{\prime}}^{n}-\delta_{s^{\prime}}^{n}}-\pair{\Delta_{s^{\prime}}^{c\cdot n}}{\frac{1}{2}}{\nu_{s^{\prime}}^{c\cdot n}-\delta_{s^{\prime}}^{c\cdot n}}.
\]
Now, according to the relations (\ref{eq:1}) one has $\delta_{s^{\prime}}^{n}+\delta_{s^{\prime}}^{c\cdot n}=1$,
hence we obtain
\begin{align}
\square_{s^{\prime}}^{n}+\square_{s^{\prime}}^{c\cdot n} & =\frac{1}{2}-\pair{\Delta_{s^{\prime}}^{n}}{\frac{1}{2}}{\nu_{s^{\prime}}^{n}-\delta_{s^{\prime}}^{n}}-\pair{1-\Delta_{s^{\prime}}^{n}}{\frac{1}{2}}{\nu_{s^{\prime}}^{n}+1-\delta_{s^{\prime}}^{n}-1}.\nonumber \\
\square_{s^{\prime}}^{n}+\square_{s^{\prime}}^{c\cdot n} & =-\frac{1}{2}\label{eq:simplification}
\end{align}
Let us focus nows on a mixed branch is of type \includegraphics[viewport=0bp 29.92126bp 299bp 95bp,width=0.08\paperwidth]{path1}$\ $.
Let $m_{c}$ be the last vertex of the branch $\AA_{c}$ where the
mixing property (\ref{eq:1}) holds. Using the simplification (\ref{eq:simplification}),
we obtain the following expression 
\begin{align*}
\Tun c{\AA}n & =\sum_{s\in\AA_{c}}\square_{s}^{n}+\square_{s}^{c\cdot n}\\
 & =\square_{r}^{n}+\square_{r}^{c\cdot n}+\sum_{s\in\AA_{m_{c}}\setminus\left\{ r\right\} }\square_{s}^{n}+\square_{s}^{c\cdot n}+\square_{m_{c}^{+}}^{n}+\square_{m_{c}^{+}}^{c\cdot n}+\sum_{s>m_{c}^{+},\ s\in\AA_{c}}\square_{s}^{n}+\square_{s}^{c\cdot n}\\
 & =-\pair 0{\frac{1}{2}}{\nu_{r}}-\pair 1{\frac{1}{2}}{\nu_{r}+1}-\frac{\left|\AA_{m_{c}}\right|-1}{2}+\frac{1}{2}-\pair 1{\frac{1}{2}}{\nu_{m_{c}^{+}}-1}-\pair 1{\frac{1}{2}}{\nu_{m_{c}^{+}}+1}+1+\Tun c{\AA^{m_{c}^{+}}}{s_{0}\cdot n}
\end{align*}
where $\AA^{m_{c}^{+}}$ is the subtree of $\AA$ whose root is $m_{c}^{+}$
and $s_{0}$ is the vertex in $\voisin r$ such that $s_{0}\geq m_{c}^{+}$.
Following Table \ref{Parit=0000E9 de nu}, we are lead to 
\begin{align*}
\Tun c{\AA}n & =-\frac{1}{2}-1-\frac{\left|\AA_{m_{c}}\right|-1}{2}+\frac{1}{2}-1+1+\Tun c{\AA^{m_{c}^{+}}}{s_{0}\cdot n}\\
 & =-\frac{1}{2}-\frac{\left|\AA_{m_{c}}\right|}{2}+\Tun c{\AA^{m_{c}^{+}}}{s_{0}\cdot n}.
\end{align*}
Applying inductively property (\ref{propriete.theta1}), we obtain
\begin{align*}
\Tun c{\AA}n & =-\frac{1}{2}-\frac{\left|\AA_{m_{c}}\right|}{2}-\frac{\left|\Aacces{m_{c}^{+}}c{}\right|}{2}-\Delta_{c}^{n}\\
 & =-\frac{\left|\Aacces{}c{}\right|}{2}-\Delta_{c}^{n},
\end{align*}
which is Property (\ref{propriete.theta1}). 

Suppose now the mixed branch is pure of type \includegraphics[viewport=0bp 29.92126bp 161bp 95bp,scale=0.2]{path\lyxdot 11}$\ $.
Then the invariant $\Tun c{\AA}n$ is written
\begin{align*}
\Tun c{\AA}n & =\sum_{s\in\AA_{c}}\square_{s}^{n}+\square_{s}^{c\cdot n}=\square_{r}^{n}+\square_{r}^{c\cdot n}+\left(\sum_{s\in\AA_{m_{c}}\setminus\left\{ r\right\} }\square_{s}^{n}+\square_{s}^{c\cdot n}\right)\\
 & =-\pair 1{\frac{1}{2}}{\nu_{r}}-\pair 0{\frac{1}{2}}{\nu_{r}+1}-\frac{\left|\AA_{m_{c}}\right|-1}{2}.
\end{align*}
According to Table (\ref{Parit=0000E9 de nu}), $\nu_{r}$ is even.
Thus 
\begin{align*}
\Tun c{\AA}n & =-1-\frac{\left|\AA_{m_{c}}\right|}{2}=-\Delta_{c}^{n}-\frac{\left|\AA_{m_{c}}\right|}{2}
\end{align*}
which is still property (\ref{propriete.theta1}). Any other type
of mixed or pure mixed branch can be treated exactly the same way. 

Now, we will prove inductively property (\ref{propriete.theta2})
from properties (\ref{propriete.theta1}) and (\ref{propriete.theta2}).
In the branch $\AA_{c_{1}},$ we denote by $r^{+}$ the successor
of $r$. Moreover, we denote by $d\in\voisin r$ such that $d\geq r^{+}.$
Depending on how start the Saito dicriticity of $\mathbb{A}$ numbered
respectively by $n$ and $c_{0}\cdot c\cdot n$, below, we expand
the expression of the invariant $\Teu{c_{0}}{c_{1}}{\AA}n$.
\begin{itemize}
\item \includegraphics[viewport=0bp 37.4046bp 118bp 98bp,scale=0.2]{path\lyxdot 20}
\begin{align*}
\Teu{c_{0}}{c_{1}}{\AA}n & =\square_{r}^{n}+\square_{r}^{c_{0}\cdot c_{1}\cdot n}+\sum_{s\in\Aacces{}{c_{1}}{}\setminus\left\{ r\right\} }\square_{s}^{n}+\square_{s}^{c_{0}\cdot c_{1}\cdot n}\\
 & =-\pair 1{\frac{1}{2}}{\nu_{r}}-\pair 1{\frac{1}{2}}{\nu_{r}+2}+\sum_{s\in\Aacces{}{c_{1}}{}\setminus\left\{ r\right\} }\square_{s}^{d\cdot n}+\square_{s}^{c_{1}\cdot d\cdot n}\\
 & =-\pair 1{\frac{1}{2}}{\nu_{r}}-\pair 1{\frac{1}{2}}{\nu_{r}+2}+\frac{1}{2}+\frac{1}{2}+\Tun{c_{1}}{\AA^{r^{+}}}{d\cdot n}\\
 & =-\pair 10{\nu_{r}}+\Tun{c_{1}}{\AA^{r^{+}}}{d\cdot n}\\
 & =-\pair 10{\nu_{r}}-\Delta_{c_{1}}^{n}-\frac{\left|\AA_{c_{1}}\right|-1}{2}=\pair{-\frac{1}{2}}{\frac{1}{2}}{\nu_{r}}-\Delta_{c_{1}}^{n}-\frac{\left|\AA_{c_{1}}\right|}{2}.
\end{align*}
\item \includegraphics[viewport=0bp 37.4046bp 116bp 98bp,scale=0.2]{path\lyxdot 21}
\begin{align*}
\Teu{c_{0}}{c_{1}}{\AA}n & =\square_{r}^{n}+\square_{r}^{c_{0}\cdot c_{1}\cdot n}+\sum_{s\in\Aacces{}{c_{1}}{}\setminus\left\{ r\right\} }\square_{s}^{n}+\square_{s}^{c_{0}\cdot c_{1}\cdot n}\\
 & =-\pair 1{\frac{1}{2}}{\nu_{r}}-\pair 0{\frac{1}{2}}{\nu_{r}+2}+\sum_{s\in\Aacces{}{c_{1}}{}\setminus\left\{ r\right\} }\square_{s}^{d\cdot n}+\square_{s}^{c_{1}\cdot n}\\
 & =-\pair 1{\frac{1}{2}}{\nu_{r}}-\pair 0{\frac{1}{2}}{\nu_{r}+2}+\frac{1}{2}+\Tnu d{c_{1}}{\AA^{r^{+}}}n\\
 & =-\frac{1}{2}+\Tnu d{c_{1}}{\AA^{r^{+}}}n\\
 & =-\frac{1}{2}\pm\frac{1}{2}-\Delta_{c}^{n}-\frac{\left|\AA_{c}\right|-1}{2}=\pm\frac{1}{2}-\Delta_{c}^{n}-\frac{\left|\AA_{c}\right|}{2}.
\end{align*}
\item \includegraphics[viewport=0bp 29.92366bp 113bp 98bp,scale=0.2]{path\lyxdot 22}
\begin{align*}
\Teu{c_{0}}{c_{1}}{\AA}n & =\square_{r}^{n}+\square_{r}^{c_{0}\cdot c_{1}\cdot n}+\sum_{s\in\Aacces{}{c_{1}}{}\setminus\left\{ r\right\} }\square_{s}^{n}+\square_{s}^{c_{0}\cdot c_{1}\cdot n}\\
 & =-\pair 0{\frac{1}{2}}{\nu_{r}}-\pair 0{\frac{1}{2}}{\nu_{r}+2}+\sum_{s\in\Aacces{}{c_{1}}{}\setminus\left\{ r\right\} }\square_{s}^{n}+\square_{s}^{c_{1}\cdot n}\\
 & =-\pair 0{\frac{1}{2}}{\nu_{r}}-\pair 0{\frac{1}{2}}{\nu_{r}+2}+\Tun{c_{1}}{\AA^{r^{+}}}n\\
 & =-\pair 01{\nu_{r}}-\Delta_{c_{1}}^{n}-\frac{\left|\AA_{c_{1}}\right|-1}{2}\\
 & =\pair{\frac{1}{2}}{-\frac{1}{2}}{\nu_{r}}-\Delta_{c_{1}}^{n}-\frac{\left|\AA_{c_{1}}\right|}{2}.
\end{align*}
\item \includegraphics[viewport=0bp 37.4046bp 116bp 98bp,scale=0.2]{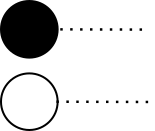}
\begin{align*}
\Teu{c_{0}}{c_{1}}{\AA}n & =\square_{r}^{n}+\square_{r}^{c_{0}\cdot c_{1}\cdot n}+\sum_{s\in\Aacces{}{c_{1}}{}\setminus\left\{ r\right\} }\square_{s}^{n}+\square_{s}^{c_{0}\cdot c_{1}\cdot n}\\
 & =-\pair 0{\frac{1}{2}}{\nu_{r}}-\pair 1{\frac{1}{2}}{\nu_{r}+2}+\sum_{s\in\Aacces{}{c_{1}}{}\setminus\left\{ r\right\} }\square_{s}^{n}+\square_{s}^{c_{1}\cdot d\cdot n}\\
 & =-\pair 0{\frac{1}{2}}{\nu_{r}}-\pair 1{\frac{1}{2}}{\nu_{r}+2}+\frac{1}{2}+\Teu d{c_{1}}{\AA^{r^{+}}}n\\
 & =-\frac{1}{2}\pm\frac{1}{2}-\Delta_{c_{1}}^{n}-\frac{\left|\AA_{c_{1}}\right|-1}{2}=\pm\frac{1}{2}-\Delta_{c_{1}}^{n}-\frac{\left|\AA_{c_{1}}\right|}{2}.
\end{align*}
\end{itemize}
That concludes the proof of Property (\ref{propriete.theta2}) for
the invariant $\Teu{c_{0}}{c_{1}}{\AA}n$. The case of the invariant
$\Tnu{c_{0}}{c_{1}}{\AA}n$ is obtained much the same way. 

To prove inductively property (\ref{Mixed.branch}) as a consequence
of all previous properties, we consider a mixed branch of any type
as in Figure \ref{fig:Mixed-branch-and}. In the sequel, the vertex
is designated by its position $k$ in the branch $k=1,\cdots$. The
$N^{\textup{th}}$ is the first for which the mixing property does
not hold. 

\begin{figure}
\begin{centering}
\includegraphics[width=0.7\textwidth]{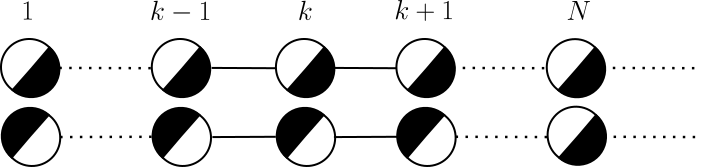}
\par\end{centering}
\caption{\label{fig:Mixed-branch-and}Mixed branch stopping from being mixed
at the $N^{th}$ vertex.}

\end{figure}

Let us denote by $\epsilon_{k}^{\star},\ \star=n,c\cdot n$ the configuration
of the $k^{th}$ vertices of the branch. Since the configuration is
supposed to be admissible, summing the two inequalities associated
to the admissibility conditions of Theorem \ref{Theoreme.Un}, we
get
\[
\epsilon_{k}^{n}+\epsilon_{k}^{c\cdot n}\geq n_{k}+2-\sum_{s\in\voisin k}\Delta_{s}^{\sigma_{k,n}}
\]
where $\sigma_{k,n}=n$ if $\Delta_{k}^{n}=0$, else $\sigma_{k,n}=c\cdot n$.
Therefore, 
\begin{equation}
\sum_{k=1}^{N-1}\epsilon_{k}^{n}+\epsilon_{k}^{c\cdot n}\geq2\left(N-1\right)+\sum_{k=1}^{N-1}n_{k}-\sum_{k=1}^{N-1}\sum_{s\in\voisin k}\Delta_{s}^{\sigma_{k,n}}.\label{ineq:}
\end{equation}
Now, we want to estimate the expression in the left of the above inequality.
Suppose first that $k=2,\cdots,N-2$. Notice that in this situation
$\delta_{k}^{\star}=\Delta_{k-1}^{\star}$, hence
\begin{align*}
\epsilon_{k}^{n}+\epsilon_{k}^{c\cdot n} & =\frac{n_{k}}{2}-\frac{\Delta_{k-1}^{n}}{2}+\underset{\nu_{k}^{n}-\Delta_{k-1}^{n}}{\left\lfloor \begin{array}{c}
\Delta_{k}^{n}\\
\frac{1}{2}
\end{array}\right.}+\sum_{s\in\voisin k}\sum_{u\in\Aacces ks{}\setminus\left\{ k\right\} }\square_{u}^{n}\\
 & +\frac{n_{k}}{2}-\frac{\Delta_{k-1}^{c\cdot n}}{2}+\underset{\nu_{k}^{c\cdot n}-\Delta_{k-1}^{c\cdot n}}{\left\lfloor \begin{array}{c}
\Delta_{k}^{c\cdot n}\\
\frac{1}{2}
\end{array}\right.}+\sum_{s\in\voisin k}\sum_{u\in\Aacces ks{}\setminus\left\{ k\right\} }\square_{u}^{c\cdot n}\\
 & =n_{k}+\frac{1}{2}+\sum_{s\in\voisin k}\sum_{u\in\Aacces ks{}\setminus\left\{ k\right\} }\square_{u}^{n}+\square_{u}^{c\cdot n}
\end{align*}

Let us denote by $k^{-}$ and $k^{+}$ the vertices in $\voisin{_{k}}$such
that $k^{-}\geq\left(k-1\right),\ k^{-}\neq k-1$ and $k^{+}\ge\left(k+1\right)$
for the order $\leq$ on the tree. Notice that $k^{-}$ may not exists
and $k^{+}$ may be equal to $k+1.$

From the previous expressions we obtain, 
\begin{align*}
\epsilon_{k}^{n}+\epsilon_{k}^{c\cdot n} & =n_{k}+\frac{1}{2}+\sum_{s\in\voisin k\setminus\left\{ k^{-},k^{+}\right\} }\sum_{u\in\Aacces ks{}\setminus\left\{ k\right\} }\square_{u}^{n}+\square_{u}^{c\cdot n}\\
 & +\sum_{u\in\Aacces k{k^{-}}{}\setminus\left\{ k\right\} }\square_{u}^{n}+\square_{u}^{c\cdot n}\\
 & +\sum_{u\in\Aacces k{k^{+}}{}\setminus\left\{ k\right\} }\square_{u}^{n}+\square_{u}^{c\cdot n}.
\end{align*}
For $s\in\voisin k\setminus\left\{ k^{-},k^{+}\right\} ,$ we are
lead to 
\begin{align*}
\sum_{u\in\Aacces ks{}\setminus\left\{ k\right\} }\square_{u}^{n}+\square_{u}^{c\cdot n} & =\frac{\left|\Aacces ks{}\setminus\left\{ k\right\} \right|}{2}+\sum_{u\in\Aacces ks{}\setminus\left\{ k\right\} }\square_{u}^{n}+\square_{u}^{s\cdot n}\\
 & =\frac{\left|\Aacces ks{}\setminus\left\{ k\right\} \right|}{2}+\Tun s{\AA^{k}}n.
\end{align*}
Hence, according to Property (\ref{propriete.theta1}), it 
\[
\sum_{u\in\Aacces ks{}\setminus\left\{ k\right\} }\square_{u}^{n}+\square_{u}^{c\cdot n}=-\Delta_{s}^{\sigma_{k,n}}.
\]
In the same way, we find 
\begin{align*}
\sum_{u\in\Aacces k{k^{-}}{}\setminus\left\{ k\right\} }\square_{u}^{n}+\square_{u}^{c\cdot n} & =\frac{1}{2}+\frac{\left|\Aacces k{k^{-}}{}\setminus\left\{ k\right\} \right|}{2}+\sum_{u\in\Aacces k{k^{-}}{}\setminus\left\{ k\right\} }\square_{u}^{a}+\square_{u}^{b}\\
 & =\frac{1}{2}+\frac{\left|\Aacces k{k^{-}}{}\setminus\left\{ k\right\} \right|}{2}+\left\{ \begin{array}{l}
\Teu k{k^{-}}{\AA^{k}}n\textup{ or }\\
\Tnu k{k^{-}}{\AA^{k}}n
\end{array}\right.
\end{align*}
where 
\[
\left\{ a,b\right\} =\left\{ \begin{array}{c}
\left\{ n,k\cdot k^{-}\cdot n\right\} \\
\left\{ k\cdot n,k^{-}\cdot n\right\} 
\end{array}\textup{ if \ensuremath{\left(\Delta_{k-1}^{n},\Delta_{k-1}^{n}\right)}}=\left\{ \begin{array}{c}
\left(0,0\right)\textup{ or }\left(1,1\right)\\
\left(0,1\right)\textup{ or }\left(1,0\right)
\end{array}\right.\right.
\]
Thus, property (\ref{propriete.theta1}) ensures that
\[
\sum_{u\in\Aacces k{k^{-}}{}\setminus\left\{ k\right\} }\square_{u}^{n}+\square_{u}^{c\cdot n}=\frac{1}{2}\pm\frac{1}{2}-\Delta_{k^{-}}^{\sigma_{k,n}}.
\]
In the same way, one can prove that
\[
\sum_{u\in\Aacces k{k^{+}}{}\setminus\left\{ k\right\} }\square_{u}^{n}+\square_{u}^{c\cdot n}=\pm\frac{1}{2}-\Delta_{k^{+}}^{\sigma_{k,n}}.
\]
Finally, for $k\in\left\{ 2,\cdots,N-2\right\} $
\begin{equation}
\epsilon_{k}^{n}+\epsilon_{k}^{c\cdot n}=n_{k}+\frac{1}{2}\pm\frac{1}{2}+\left\{ \begin{array}{cl}
\frac{1}{2}\pm\frac{1}{2} & \textup{ if }k^{-}\textup{ exists}\\
0 & \textup{else}
\end{array}\right.-\sum_{s\in\voisin k}\Delta_{s}^{\sigma_{k,n}}\label{eq:22}
\end{equation}
For $k=1$ the computation is slightly different bu we obtain
\begin{align}
\epsilon_{1}^{n}+\epsilon_{1}^{c\cdot n} & =n_{1}+\frac{1}{2}+\underset{\nu_{1}}{\left\lfloor \begin{array}{c}
\Delta_{1}^{n}\\
\Delta_{1}^{c\cdot n}
\end{array}\right.}\pm\frac{1}{2}-\sum_{s\in\voisin 1}\Delta_{s}^{\sigma_{1,n}}.\label{eq:23}
\end{align}
Finally for $k=N-1,$ we can write 
\begin{align*}
\epsilon_{N-1}^{n}+\epsilon_{N-1}^{c\cdot n} & =\frac{n_{N-1}}{2}-\frac{\Delta_{N-2}^{n}}{2}+\underset{\nu_{N}^{n}-\Delta_{N-2}^{n}}{\left\lfloor \begin{array}{c}
\Delta_{N-1}^{n}\\
\frac{1}{2}
\end{array}\right.}+\frac{n_{N-1}}{2}-\frac{\Delta_{N-2}^{c\cdot n}}{2}+\underset{\nu_{N}^{c\cdot n}-\Delta_{N-2}^{c\cdot n}}{\left\lfloor \begin{array}{c}
\Delta_{N-1}^{c\cdot n}\\
\frac{1}{2}
\end{array}\right.}\\
 & +\sum_{s\in\voisin{N-1}}\sum_{u\in\Aacces{N-1}s{}\setminus\left\{ N-1\right\} }\square_{u}^{n}+\square_{u}^{c\cdot n}\\
 & =n_{N-1}+\frac{1}{2}+\sum_{u\in\Aacces{\left(N-1\right)}{\left(N-1\right)^{+}}{}\setminus\left\{ \left(N-1\right)\right\} }\square_{u}^{n}+\square_{u}^{c\cdot n}\\
 & +\sum_{s\in\voisin{N-1},s\neq\left(N-1\right)^{+}}\sum_{u\in\Aacces{N-1}s{}\setminus\left\{ N-1\right\} }\square_{u}^{n}+\square_{u}^{c\cdot n}\\
\end{align*}

Now, we focus on the last term of the above sum : we have
\begin{align*}
\sum_{u\in\Aacces{N-1}{\left(N-1\right)^{+}}{}\setminus\left\{ N-1\right\} }\square_{u}^{n}+\square_{u}^{c\cdot n} & =\frac{\Delta_{N-1}^{n}}{2}+\frac{\Delta_{N-1}^{c\cdot n}}{2}-\underset{\nu_{N}^{n}-\Delta_{N-1}^{n}}{\left\lfloor \begin{array}{c}
\Delta_{N}^{n}\\
\frac{1}{2}
\end{array}\right.}-\underset{\nu_{N}^{c\cdot n}-\Delta_{N-1}^{c\cdot n}}{\left\lfloor \begin{array}{c}
\Delta_{N}^{c\cdot n}\\
\frac{1}{2}
\end{array}\right.}\\
 & +\frac{\Delta_{N}^{n}}{2}+\frac{\Delta_{N}^{c\cdot n}}{2}\\
 & -\left(\frac{\Delta_{N}^{n}}{2}+\frac{\Delta_{N}^{c\cdot n}}{2}\right)+\sum_{u\in\Aacces{\left(N-1\right)}{\left(N-1\right)^{+}}{}\setminus\left\{ \left(N-1\right)\right\} ,\ u\neq N}\square_{u}^{n}+\square_{u}^{c\cdot n}.
\end{align*}
Since $\Delta_{N}^{n}=\Delta_{N}^{c\cdot n}$, $\nu_{N}^{n}=\nu_{N}^{c\cdot n}-1$
and $\Delta_{N-1}^{n}+\Delta_{N-1}^{c\cdot n}=1$, we are lead to
the expression
\[
\sum_{u\in\Aacces{N-1}{\left(N-1\right)^{+}}{}\setminus\left\{ N-1\right\} }\square_{u}^{n}+\square_{u}^{c\cdot n}=\frac{1}{2}-\underset{\nu_{N}-\Delta_{N-1}^{n}}{\left\lfloor \begin{array}{c}
\Delta_{N}^{n}\\
1-\Delta_{N}^{n}
\end{array}\right.}-\Delta_{\left(N-1\right)^{+}}^{\sigma_{\left(N-1\right),n}}
\]
Finally, we find
\begin{align}
\epsilon_{N-1}^{n}+\epsilon_{N-1}^{c\cdot n} & =n_{N-1}+1-\underset{\nu_{N}-\Delta_{N-1}^{n}}{\left\lfloor \begin{array}{c}
\Delta_{N}^{n}\\
1-\Delta_{N}^{n}
\end{array}\right.}\label{eq:24}\\
 & +\left\{ \begin{array}{cl}
\frac{1}{2}\pm\frac{1}{2} & \textup{ if \ensuremath{\left(N-1\right)^{-}}exists}\\
0 & \textup{else}
\end{array}\right.-\sum_{s\in\voisin{_{N-1}}}\Delta_{s}^{\sigma_{N-1,n}}\nonumber 
\end{align}

Summing the equations (\ref{eq:22}), (\ref{eq:23}) and (\ref{eq:24})
yields 
\[
\sum_{k=1}^{N-1}\epsilon_{k}^{n}+\epsilon_{k}^{c\cdot n}=\underset{\nu_{1}}{\left\lfloor \begin{array}{c}
\Delta_{1}^{n}\\
\Delta_{1}^{c\cdot n}
\end{array}\right.}-\underset{\nu_{N}-\Delta_{N-1}^{n}}{\left\lfloor \begin{array}{c}
\Delta_{N}^{n}\\
1-\Delta_{N}^{n}
\end{array}\right.}+\underbrace{\left(\cdots\right)}_{\textup{ at most }2N-3}+\sum_{k=1}^{N-1}n_{k}-\sum_{k=1}^{N-1}\sum_{s\in\voisin k}\Delta_{s}^{\sigma_{k,n}}.
\]
Combining with the inequality (\ref{ineq:}), we obtain
\[
\underset{\nu_{1}}{\left\lfloor \begin{array}{c}
\Delta_{1}^{n}\\
\Delta_{1}^{c\cdot n}
\end{array}\right.}-\underset{\nu_{N}-\Delta_{N-1}^{n}}{\left\lfloor \begin{array}{c}
\Delta_{N}^{n}\\
1-\Delta_{N}^{n}
\end{array}\right.}\geq1.
\]
This inequality induces all the properties presented in Table \ref{Parit=0000E9 de nu}. 

To prove the statements in Table \ref{pure.branch}, we add one white
component at the end of each pure mixed branches, providing thus standard
mixed branches. Numbering the vertices of these branches $1,\cdots,N,N+1$,
the $N+1^{\textup{th}}$ being the added one and setting $n_{N+1}=0$,
we obtain two mixed branches numbered respectively by $n$ and $\left(N+1\right)\cdot n$
whose dicricities are Saito. Thus, the computations performed for
mixed branches yield
\[
\underset{\nu_{1}}{\left\lfloor \begin{array}{c}
\Delta_{1}^{n}\\
\Delta_{1}^{c\cdot n}
\end{array}\right.}-\underset{0-\Delta_{N}^{n}}{\left\lfloor \begin{array}{c}
1\\
0
\end{array}\right.}\geq1.
\]
Thus if $\Delta_{N}^{n}=0$ the above inequality is impossible ; that
excludes the two last cases of Table (\ref{pure.branch}). If $\Delta_{N}^{n}=1$,
then the inequality reduces to 
\[
\underset{\nu_{1}}{\left\lfloor \begin{array}{c}
\Delta_{1}^{n}\\
\Delta_{1}^{c\cdot n}
\end{array}\right.}\geq1,
\]
which implies the two first cases of the table.\emph{ }Finally, suppose
that the mixed branch is reduced to a single couple of vertices\emph{
}and starts with \includegraphics[scale=0.07]{path\lyxdot 132}. Assume
that $\nu_{r}^{n}$ is even. We can write, 
\[
\epsilon_{r}^{n}=\frac{n_{r}}{2}+\pair 0{\frac{1}{2}}{\nu_{r}^{n}}+\sum_{v\in\voisin c}\square_{v}.
\]
Hence, we deduce that 
\begin{align*}
\epsilon_{r}^{n}+1 & =\frac{n_{r}+1}{2}+\pair 0{\frac{1}{2}}{\nu_{r}^{n}+1}+\sum_{v\in\voisin c}\square_{v}.\\
 & =\frac{n_{r}+1}{2}+\pair 0{\frac{1}{2}}{\nu_{r}^{r\cdot n}}+\sum_{v\in\voisin c}\square_{v}.
\end{align*}
Since, $\epsilon^{n}$ is the configuration of the Saito dicriticity
for the numbering $n$, we get 
\[
\epsilon_{r}^{n}\geq2-\sum_{s\in\voisin r}\Delta_{s}^{n}
\]
and consequently, 
\[
\epsilon_{r}^{n}+1\geq2-\sum_{s\in\voisin r}\Delta_{s}^{n}.
\]
Therefore, the dicriticity $\Delta^{n}$ keeps on being Saito for
the tree $\AA$ numbered by $r\cdot n.$ Since, this dicriticity is
unique, we get $\Delta_{r}^{r\cdot n}=1$ which contradicts the hypothesis
of a mixed branch. Finally, $\nu_{r}^{n}$ has to be odd. In the same
way, suppose the mixed branch is reduced to \includegraphics[scale=0.07]{path\lyxdot 133}
and $\nu_{r}^{n}$ is odd. The arguments are the same as above and
from the following computations
\begin{align*}
\epsilon_{r}^{n} & =\frac{n_{r}}{2}+\pair 1{\frac{1}{2}}{\nu_{r}^{n}}+\sum_{v\in\voisin c}\square_{v}\geq n_{r}\\
\epsilon_{r}^{n}+1 & =\frac{n_{r}+1}{2}+\pair 1{\frac{1}{2}}{\nu_{r}^{r\cdot n}}+\sum_{v\in\voisin c}\square_{v}\geq n_{r}+1
\end{align*}
we get a contradiction. Hence, $\nu_{r}^{n}$ is even. This concludes
the proof of property (\ref{Mixed.branch}).

It remains to prove property (\ref{component.dicritique}). Let $\mathbb{K}$
be the connected component of $r$ in the sub-graph $\mathbb{A}\setminus\left\{ \left.c\in\AA\right|\Delta_{c}^{n}=0\right\} $.
If $\mathbb{K}=\emptyset,$ the property is proved by induction on
$\left|\AA\right|$. If not, suppose that there exists $s\in\mathbb{K}$
such that $n_{s}>0.$ Then, since $\Delta_{s}^{n}=1,$ the admissibility
condition ensures that $\epsilon_{s}^{n}\geq n_{s}>0$ which is the
property. Finally, we suppose that for any $s\in\mathbb{K},$ $n_{s}=0$.
Assume also that for any $s\in\mathbb{K}$, $\epsilon_{s}^{n}=0.$
For any $s\in\voisin r$, we consider $k_{s}\in\Aacces{}s{}$ the
minimal vertex such that $\Aacces{k_{s}}s{}$ is in $\mathbb{K}.$
Now, one can write 
\begin{align}
0=\epsilon_{r}^{n} & =-\square_{r}+\sum_{v\in\voisin r}\sum_{s\in\Aacces{}v{}\setminus\left\{ r\right\} }\square_{s}^{n}.\label{eq:compliquee}\\
 & =\pair 1{\frac{1}{2}}{\nu_{r}^{n}}+\sum_{v\in\voisin r}\sum_{\Aacces{}{k_{v}^{-1}}{}\setminus\left\{ r\right\} }\square_{s}^{n}\nonumber \\
 & +\sum_{v\in\voisin r}\sum_{\Aacces{k_{v}}v{}\setminus\left\{ r\right\} }\square_{s}^{n}\nonumber 
\end{align}

Now, extracting the intermediary sum in the expression above yields
\begin{align*}
\sum_{\Aacces{}{k_{v}^{-1}}{}\setminus\left\{ r\right\} }\square_{s}^{n} & =\sum_{\Aacces{}{k_{v}^{-1}}{}\setminus\left\{ r\right\} }\frac{\delta_{s}^{n}}{2}-\pair{\Delta_{s}^{n}}{\frac{1}{2}}{\star}\\
 & =\sum_{\Aacces{}{k_{v}^{-1}}{}\setminus\left\{ r\right\} }\frac{\delta_{s}^{n}}{2}-\sum_{\Aacces{}{k_{v}^{-1}}{}\setminus\left\{ r\right\} }\pair{\Delta_{s}^{n}}{\frac{1}{2}}{\star}\\
 & =\frac{\delta_{r^{+}}^{n}}{2}+\sum_{\Aacces{}{k_{v}^{-1}}{}\setminus\left\{ r,r^{+}\right\} }\frac{\delta_{s}^{n}}{2}-\sum_{\Aacces{}{k_{v}^{-1}}{}\setminus\left\{ r\right\} }\pair{\Delta_{s}^{n}}{\frac{1}{2}}{\star}\\
 & =\frac{\delta_{r^{+}}^{n}}{2}+\sum_{\Aacces{}{k_{v}^{-2}}{}\setminus\left\{ r\right\} }\frac{\delta_{s^{+}}^{n}}{2}-\left(\sum_{\Aacces{}{k_{v}^{-2}}{}\setminus\left\{ r\right\} }\pair{\Delta_{s}^{n}}{\frac{1}{2}}{\star}\right)-\pair{\Delta_{k_{v}^{-1}}^{n}}{\frac{1}{2}}{\star}.
\end{align*}
where $r^{+}$ is the successor of $r$ in $\Aacces{}v{}$, $k_{v}^{-i}$
the predecessor of $k_{v}^{-i+1}.$ Since $\Delta_{r}^{n}=1,$ $\delta_{r^{+}}^{n}=1$
and $\delta_{s^{+}}^{n}=1+\Delta_{s}^{n}$. Hence, we obtain
\begin{equation}
\sum_{\Aacces r{k_{v}^{-1}}{}\setminus\left\{ r\right\} }\square_{s}^{n}=\underbrace{\frac{1}{2}-\pair{\Delta_{k_{v}^{-1}}^{n}}{\frac{1}{2}}{\star}}_{A}+\underbrace{\sum_{\Aacces r{k_{v}^{-2}}{}\setminus\left\{ r\right\} }\pair{\frac{1-\Delta_{s}^{n}}{2}}{\frac{\Delta_{s}^{n}}{2}}{\star}}_{\geq0}.\label{Equation.pour.composante}
\end{equation}
Now, since $k_{v}\in\mathbb{K}$, then 
\[
0=\epsilon_{k_{v}}^{n}=-\square_{k_{v}}^{n}+\sum_{v\in\voisin{k_{v}}}\sum_{s\in\Aacces{k_{v}}v{}\setminus\left\{ k_{v}\right\} }\square_{s}^{n}.
\]
Therefore 
\[
\square_{k_{v}}^{n}=\sum_{v\in\voisin{k_{v}}}\sum_{s\in\Aacces{k_{v}}v{}\setminus\left\{ k_{v}\right\} }\square_{s}^{n}.
\]

If $\Delta_{k_{v}^{-1}}^{n}=0$ then, in expression (\ref{Equation.pour.composante})
$A=\frac{1}{2}-\pair{\Delta_{k_{v}^{-1}}^{n}}{\frac{1}{2}}{\star}\geq0.$
If $\Delta_{k_{v}^{-1}}^{n}=1$, then by construction of $k_{v}$,
$k_{v}^{-1}\notin\voisin{k_{v}}$. In the latter case, there exists
$s\in\voisin{k_{v}}$ and $c\in\Aacces{k_{v}}s{}$ with $\parent c=\left\{ k_{v},k_{v}^{-1}\right\} .$
Thus 
\[
\delta_{c}=\frac{1}{2}+\frac{1}{2}
\]
which comes to compensate the fact that in relation (\ref{Equation.pour.composante}),
$A$ might be equal to $\frac{-1}{2}.$ Finally, if $s\in\Aacces{k_{v}}v{}\setminus\left\{ r\right\} $,
$s\neq k_{v}$ then as before, 
\[
\square_{s}^{n}=\sum_{v\in\voisin s}\sum_{u\in\Aacces sv{}\setminus\left\{ s\right\} }\square_{u}^{n}.
\]
Doing so step by step, from (\ref{eq:compliquee}), we are lead to
an expression of the form 
\[
0=\pair 1{\frac{1}{2}}{\nu_{r}^{n}}+\underbrace{\left(\cdots\right)}_{\geq0}
\]
which is impossible. That concludes the proof of property (\ref{component.dicritique})
and, at the same time, the proof of Theorem \ref{Theoreme.Un}.
\end{proof}

\section{Saito foliations of a germ of curve and its deformation.}

The $\mathcal{O}_{\left(\mathbb{C}^{2},0\right)}-$module $\textup{Der}\left(\log C\right)$
of germs of vector fields tangent to a germ of curve $C\subset\left(\mathbb{C}^{2},0\right)$
has been introduced as a particular case of a far more general object
by K. Saito in \cite{MR586450}. From now on, we are interesting in
the valuations of the vector fields in this module, for the standard
valuation $\nu$ defined by 
\[
\nu\left(a\partial_{x}+b\partial_{y}\right)=\min\left(\nu\left(a\right),\nu\left(b\right)\right),\ a,b\in\mathbb{C}\left\{ x,y\right\} ,
\]
where 
\[
\nu\left(\sum_{i,j}a_{ij}x^{i}y^{j}\right)=\min_{a_{ij}\neq0}\left\{ i+j\right\} .
\]

In particular, we define \emph{the number of Saito of $C$ }by 
\[
\mathfrak{s}_{C}=\min_{X\in\textup{Der}\left(\log C\right)}\nu\left(X\right)
\]
A vector field tangent to $C$ is said \emph{optimal }if its valuation
is equal to the Saito number of $C.$ 

Let $\pi$ be the blowing-up of the singular point of $C.$ At any
singular point $s$ of the total transform $\pi^{-1}\left(C\right)$,
the strict transform $X^{\pi}$ of $X$ leaves invariant the strict
transfotm $C^{\pi}$ and maybe the exceptional divisor of $\pi$.
When the latter occurs, $X$ is said \emph{non dicritical}. Otherwise,
it is said \emph{dicritical}. The vector field $X^{\pi}$ may not
be optimal for the germ of $C^{\pi}$ at $s$ although $X$ is optimal
for $C.$ When the optimality property is preserved along the desingularization
process of $C$, we said that $X$ is \emph{Saito} for $C.$ More
precisely, we consider the following inductive definition :
\begin{defn}
$X$ is \emph{Saito} for $C$ when $X$ is optimal for $C$ and when
$X^{\pi}$ is Saito for each germ of $D\cup C^{\pi}$ at any of its
singular points where 
\begin{itemize}
\item $D=\pi^{-1}\left(0\right)$ if $X$ is not dicritical, 
\item $D=\emptyset$, otherwise. 
\end{itemize}
\end{defn}

To initiate the definition, we require that if $C$ is regular, then
$\nu\left(X\right)=0$ and if $C$ is the union of two transversal
regular curves, then $\nu\left(X\right)=1.$

The goal of the current section is to prove the existence of a curve
$C^{\prime}$ topologically equivalent to $C$ that admits a Saito
foliation. To do so, we are going to construct a foliation using gluing
techniques described in \cite{alcides}. The elementary pieces of
this gluing are semi-local models for Saito foliations introduced
just below. The result of the first section will provide a global
data prescribing the gluing. The obtained foliation will be studied
from the point of view of deformations and the curve $C^{\prime}$
will be found as an invariant curve of a generic deformation of the
constructed foliation. 

\subsection{Semi-local models for Saito foliations.}

First, let us describe the two families of semi-local models for Saito
foliations. These models are said to be semi-local because they are
defined in the neighborhood of a compact divisor embedeed in a surface. 

Let $\mathcal{M}_{p}$ be the germ of neighborhood of the divisor,
given locally by $x_{1}=0$ and $y_{2}=0$, in the dimension $2$
manifold defined by the disjont union of two charts

\[
\left(\mathbb{C}^{2},\left(x_{1},y_{1}\right)\right)\coprod\left(\mathbb{C}^{2},\left(x_{2},y_{2}\right)\right)
\]
with the identification $y_{2}=y_{1}^{p}x_{1}\quad x_{2}=\frac{1}{y_{1}},\ p\geq0.$
The divisor $\left\{ x_{1}=y_{2}=0\right\} $ is a regular rational
compact curve embedded in $\mathcal{M}_{p}$ with negative self-intersection
equal to $-p$. 

\subsubsection{The dicritical model.}

The manifold $\mathcal{M}_{p}$ can be foliated by the foliation $\mathcal{R}_{p,N}$
given in coordinates $\left(x_{1},y_{1}\right)$ by the $1-$form
\begin{equation}
\textup{d}x_{1}+\prod_{i=1}^{N}\left(y_{1}-i\right)\textup{d}y_{1}.\label{eq:dic}
\end{equation}
This foliation is transverse to the compact divisor except at the
points given in coordinates $\left(x_{1},y_{1}\right)$ by $\left(0,i\right),$
$i=1,\ldots,N$ where it is tangent at order $1.$ Using the changes
of coordinates form $\left(x_{1},y_{1}\right)$ to $\left(x_{2},y_{2}\right)$,
we can see that the foliation is regular and transverse to the compact
divisor at $+\infty.$ Note that we have
\[
\sum_{p\in\left\{ x_{1}=0\right\} }\textup{Tan}\left(\mathcal{R}_{p,N},\left\{ x_{1}=0\right\} ,p\right)=N
\]
where $\textup{Tan}$ is an index introduced in particular in \cite{hertlingfor}.

\subsubsection{The non dicritical model.}

The manifold $\mathcal{M}_{p}$ can also be foliated by a foliation
given by the $1-$form $\mathcal{G}_{p,N,\Lambda}$, $\Lambda=\left(\lambda_{1},\ldots,\lambda_{N}\right)$
written in the coordinates $\left(x_{1},y_{1}\right)$ 
\begin{equation}
\frac{dx_{1}}{x_{1}}+\sum_{i=1}^{N}\lambda_{i}\frac{\textup{d}y_{1}}{y_{1}-i}.\label{eq:nondic}
\end{equation}
with the following condition, known as the Camacho-Sad relation \cite{separatrice},
\begin{equation}
\sum_{i=1}^{N}\lambda_{i}=p.\label{eq:camacho}
\end{equation}

This foliation leaves invariant the divisor $x_{1}=0$ and the relation
above ensure that it is regular at $+\infty.$ By construction, for
any $i,$ we get 
\[
\textup{CS}\left(\mathcal{G},\left\{ x_{1}=0\right\} ,i\right)=\lambda_{i}.
\]

where $\textup{CS}$ is the so-called Camacho-Sad index \cite{separatrice}.
Moreover, it follows that 
\[
\sum_{p\in\left\{ x_{1}=0\right\} }\textup{Ind}\left(\mathcal{R}_{p,N},\left\{ x_{1}=0\right\} ,p\right)=N
\]
where $\textup{Ind}$ is the second index introduced in \cite{hertlingfor}. 

Figure \ref{fig:Local-model-of} presents the topology of $\mathcal{R}$
and $\mathcal{G}$. 
\begin{figure}
\begin{centering}
\includegraphics[width=4.5cm]{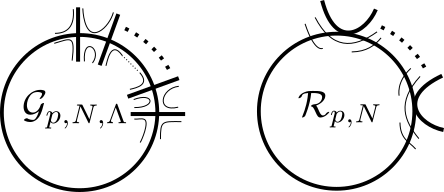}
\par\end{centering}
\caption{\label{fig:Local-model-of}Local models for Saito foliations.}
\end{figure}

\subsection{Gluing local models. }

Let $E$ be the process of desingularization of $C.$ Let $\AA$ be
the dual tree of the exceptional divisor $E^{-1}\left(0\right).$
The map $E$ is a composition of elementary blowing-ups that we denote
\[
E=\bigcirc_{s\in\AA}E_{s}.
\]
Here $E_{s}$ is the elementary blowing-up whose exceptional divisor
is the component $s$. For any $c$, the notation $\star_{c}^{E}$
will refer to the germ at the point leading to the component $c$
of the strict transform of $\star$ by the sub-process $\bigcirc_{s\in\AA_{c}\setminus\left\{ c\right\} }E_{s}$
where $\AA_{c}$ is the access tree from $r$ to $c$, as defined
in the previous section. 

For a germ of vector field $X$ (or its associated germ of foliation
$\mathcal{F}$) and $s\in\AA$, we will set $\Delta_{s}^{X\ \left(\textup{or }\mathcal{F}\right)}=1$
if $X_{s}^{E}$ is non dicritical, otherwise, $\Delta_{s}^{X}=0.$
It defines a dicriticity on $\AA.$
\begin{prop}
\label{prop:presque.optimal}There exists $C^{\prime}$ topologically
equivalent to $C$ such that there exists $X\in\textup{Der}\left(\textup{log}C^{\prime}\right)$
satisfying the following : for any $s\in\AA$ 
\[
\nu\left(X_{s}^{E}\right)=\frac{\nu\left(C_{s}^{E}\right)+\delta_{s}^{X}}{2}-\pair{1-\Delta_{s}^{X}}{\frac{1}{2}}{\nu\left(C_{s}^{E}\right)+\delta_{s}^{X}}
\]
\end{prop}

\begin{proof}
The process $E$ of desingularization of $C$ induces an numbered
ordered tree $\AA$ as defined in the previous section. The tree $\AA$
is the dual tree of the exceptional divisor of $E$ ; the order is
the one induced by the process it-self ; the numbering $n$ is setting
as follows : $n_{s}$ is equal to the number of component of the strict
transform $C^{E}$ attached to $s.$ Consider the associated Saito
dicriticity $\Delta^{n}$ and configuration $\epsilon^{n}$ given
by Theorem \ref{Theoreme.Un}. 

Using a result of A.-L. Neto \cite{alcides} of construction of singular
foliations in dimension $2$ from elementary elements, we are going
to construct a foliation from the data of $\Delta^{n}$ and $\varepsilon^{n}$
by gluing semi-local models. The matrix $\mathbb{P}$ being the proximity
matrix of $\AA$ it is known that $\mathbb{P}^{t}\mathbb{P}$ is the
intersection matrix $I$ of $E^{-1}\left(0\right)$ embedded in its
neighborhood \cite{MR2107253}. 

To $s\in\AA$ with $\Delta_{s}^{n}=1$, we associate the semi-local
model $\mathcal{G}_{I_{s,s},\epsilon_{s}^{n}+\left|\left\{ c\in\voisin s,\Delta_{c}^{n}=1\right\} \right|,\Lambda_{s}}$
where 
\[
\Lambda_{s}=\left(\lambda_{1},\cdots,\lambda_{\epsilon_{s}^{n}},\lambda_{s,c_{1}},\cdots,\lambda_{s,c_{\left|\left\{ c\in\voisin s,\Delta_{c}^{n}=1\right\} \right|}}\right)/
\]
The only obstruction for such a semi-local construction is the Camacho-Sad
relation 
\begin{equation}
\sum_{i=1}^{\epsilon_{s}^{n}}\lambda_{i}+\sum_{i\in\left\{ c\in\voisin s,\Delta_{c}^{n}=1\right\} }\lambda_{s,i}=-I_{s,s}\label{eq:Camacho}
\end{equation}
To $s\in\AA$ with $\Delta_{s}^{n}=0,$ we associate the semi-local
model 
\[
\mathcal{R}_{I_{s,s},\epsilon_{s}^{n}-2+\sum_{c\in\mathfrak{v_{s}}}\Delta_{c}^{n}}.
\]
 Since $\Delta^{n}$ is the Saito dicriticity, the admissibility condition
yields the inequality
\[
\epsilon_{s}^{n}-2+\sum_{c\in\mathfrak{v_{s}}}\Delta_{c}^{n}\geq0,
\]
so that, the definition of the model does make sense. 

From \cite{alcides}, all these semi-local models can be glued together
by gluing maps following the edges of $\mathbb{A}$ provided that
at any intersection point of two components $s$ and $s^{\prime}$
with $\Delta_{s}^{n}=\Delta_{s^{\prime}}^{n}=1$, the following relation
is satisfied
\begin{equation}
\lambda_{s,s^{\prime}}\cdot\lambda_{s^{\prime},s}=1.\label{eq:CS-1}
\end{equation}
Property (\ref{component.dicritique}) of Theorem \ref{Theoreme.Un}
ensures that, along any connected component $\mathbb{K}$ of $\mathbb{A}\setminus\left\{ \left.s\in\mathbb{A}\right|\Delta_{s}^{n}=1\right\} $,
no incompatiblity will occur between the relations (\ref{eq:Camacho})
and (\ref{eq:CS-1}). Indeed, along $\mathbb{K}$ the number of induced
relations is 
\[
\sharp\textup{vertices}\left(\mathbb{K}\right)+\sharp\textup{edges}\left(\mathbb{K}\right).
\]
However, the number of variables involved in the mentioned relations
is 
\[
\sum_{s\in\mathbb{K}}\epsilon_{s}^{n}+\left|\left\{ c\in\voisin s,\Delta_{c}^{n}=1\right\} \right|.
\]
Following Property (\ref{component.dicritique}) the above number
of variables satisfies
\begin{align*}
\sum_{s\in\mathbb{K}}\epsilon_{s}^{n}+\left|\left\{ c\in\voisin s,\Delta_{c}^{n}=1\right\} \right| & \geq1+\sum_{s\in\mathbb{K}}\left|\left\{ c\in\voisin s,\Delta_{c}^{n}=1\right\} \right|=2\sharp\textup{vertices}\left(\mathbb{K}\right)-1\\
 & \geq\sharp\textup{vertices}\left(\mathbb{K}\right)+\sharp\textup{edges}\left(\mathbb{K}\right)
\end{align*}
since $\sharp\textup{vertices}\left(\mathbb{K}\right)=\sharp\textup{edges}\left(\mathbb{K}\right)+1.$
Therefore the system of equations, union of (\ref{eq:Camacho}) and
(\ref{eq:CS-1}), has always a solution. Note that these solutions
can be chosen to be rational numbers. 

The gluing leads to a foliation defined in a neighborhood of a compact
divisor $\mathcal{D}$, union of $\left|\AA\right|$ regular rational
curves, with same intersection matrix as the one of the exceptional
divisor of $E.$ According to a classical result of H. Grauert \cite{grauerthans},
the neighborhood of $\mathcal{D}$ is analytically equivalent to the
neighborhood of the exceptional divisor of some blowing-up process
$E^{\prime}$ with same dual graph as $E.$ The latter neighborhood
is foliated by a foliation $\mathcal{F}^{\prime}$ that can be contracted
by $E^{\prime}$ in a foliation $\mathcal{F}.$ 

For any component $s\in\AA$, either $\Delta_{s}^{n}=0$ and $\mathcal{F}^{\prime}$
is generically transverse to $s$. Then, we choose arbitraly $n_{s}$
regular and transverse invariant curves attached to s. Or $\Delta_{s}^{n}=1$
and $\mathcal{F}^{\prime}$ locally given by (\ref{eq:nondic}) leaves
invariant at least $n_{s}$ regular and transverse curves attached
to $s$: indeed, $\Delta^{n}$ being Saito, we have $\epsilon_{s}^{n}\geq n_{s}.$
The union of all these curves yields a curve $C^{\prime}$ whose desingularization
process has for numbered dual graph the tree $\AA$. Since the number
of component of the strict transform by $E$ of $C^{\prime}$ has
exactly $n_{s}$ components attached to $s$, $C^{\prime}$ and $C$
are topologically equivalent \cite{zariskitop}. In the sequel, for
the sake of simplicity, we denote $C^{\prime}$ simple by $C$. According
to \cite[Theorem 3]{hertlingfor}, we have
\[
\nu\left(\mathcal{F}\right)+1=\sum_{s\in\AA}\rho_{s}\times\left\{ \begin{array}{cl}
-\left|\left\{ c\in\voisin s,\Delta_{c}^{n}=1\right\} \right|+\sum_{p\in s}\textup{Ind}\left(\mathcal{F}^{E},s,p\right) & \textup{ if }\Delta_{s}^{n}=1\\
2-\left|\left\{ c\in\voisin s,\Delta_{c}^{n}=1\right\} \right|+\sum_{p\in s}\textup{Tan}\left(\mathcal{F}^{E},s,p\right) & \textup{ if }\Delta_{s}^{n}=0
\end{array}\right.
\]
 In our construction, the definition of the semi-local models induces
the relations 
\begin{align*}
\sum_{p\in s}\textup{Ind}\left(\mathcal{F}^{E},s,p\right) & =\epsilon_{s}^{n}+\left|\left\{ c\in\voisin s,\Delta_{c}^{n}=1\right\} \right|\\
\sum_{p\in s}\textup{Tan}\left(\mathcal{F}^{E},s,p\right) & =\epsilon_{s}^{n}-2+\left|\left\{ c\in\voisin s,\Delta_{c}^{n}=1\right\} \right|.
\end{align*}
Moreover, by construction for any $s\in\AA$, we find 
\[
\Delta_{s}^{n}=\Delta_{s}^{\mathcal{F}},\ \delta_{s}^{n}=\delta_{s}^{\mathcal{F}}.
\]
Thus, since the configuration $\epsilon^{n}$ satisfies the system
(\ref{configuration}) of Theorem \ref{Theoreme.Un}, the valuation
of $\mathcal{F}$ can be expressed as follows 
\[
\nu\left(\mathcal{F}\right)=\sum_{s\in\AA}\rho_{s}\epsilon_{s}^{n}=\frac{\nu_{1}^{n}}{2}-\pair{1-\Delta_{1}^{n}}{\frac{1}{2}}{\nu_{1}^{n}}=\frac{\nu\left(C\right)}{2}-\pair{1-\Delta_{1}^{\mathcal{F}}}{\frac{1}{2}}{\nu\left(C\right)}.
\]
Doing the same remark along the whole process of blowing-ups of $C$,
we obtain, for any $s\in\AA$, 
\begin{align}
\nu\left(\mathcal{F}_{s}^{E}\right) & =\frac{\nu\left(C_{s}^{E}\right)+\delta_{s}^{\mathcal{F}}}{2}-\pair{1-\Delta_{s}^{\mathcal{F}}}{\frac{1}{2}}{\nu\left(C_{s}^{E}\right)+\delta_{s}^{\mathcal{F}}}\label{eq:valuation.saito}
\end{align}
\end{proof}

\subsection{Deformation of $\mathcal{F}$. }

In the previous section, we obtained a foliation $\mathcal{F}$ leaving
invariant a curve $C$ whose valuations satisfy the relations described
in Proposition \ref{prop:presque.optimal}. However, we cannot still
claim that a vector field $X$ defining $\mathcal{F}$ is neither
optimal nor Saito, since the curve $C$ could be \emph{special }in
its moduli space and admit a tangent vector field with small valuations.
In order to overcome this difficulty, we are going to prove that $\mathcal{F}$
can be put in a deformation \emph{weakly equisingular} that follows
a deformation of toward generic elements of $\textup{Mod}\left(C\right)$,
for which lower bound for Saito numbers is known. To implement this
strategy, we will gather material from \cite{MR2422017,YoyoBMS,genzmer2020saito,Genzmer1,Gomez}.
\begin{thm}
There exists $C^{\prime}$ in the moduli space of $C$ such that $C^{\prime}$admits
a Saito vector field $X$ further satisfying for any $s\in\AA$
\[
\nu\left(X_{s}^{E}\right)=\frac{\nu\left(C_{s}^{E}\right)+\delta_{s}^{X}}{2}-\pair{1-\Delta_{s}^{X}}{\frac{1}{2}}{\nu\left(C_{s}^{E}\right)+\delta_{s}^{X}}
\]
\end{thm}

\begin{proof}
Let $E$ be the desingularization process of $C$. Denote by $\Omega$
the volume form 
\[
\text{\ensuremath{\Omega}}=E^{\star}\left(\textup{d}x\wedge\textup{d}y\right).
\]

Let $\mathfrak{X}$ be the global vector field $\mathfrak{X}=E^{\star}\left(\frac{X}{f}\right)$
where $X$ is a vector field defining $\mathcal{F}$ and $f$ is \emph{a
balanced equation of the separatricies of $X$}, as introduced in
\cite[Definition 1.2]{Genzmer1}. Following \cite[Proposition 18]{YoyoBMS},we
associate to $\mathfrak{X}$ the following divisor 
\begin{equation}
D_{\mathfrak{X}}=2\left(\left(f=0\right)^{E}-\left(f=\infty\right)^{E}\right)-C^{E}+\overline{D}\label{eq:divisor}
\end{equation}
defined in the total space of $E$. Here, $\overline{D}$ is the union
of components of $D$ invariant by $\mathfrak{X}$. Let us consider
$\mathbb{F}$ the sheaf based upon $D$ of $\mathcal{O}-$modules
of vector fields tangent to the foliation given by $\mathfrak{X}$
and $\Theta$ the sheaf based on $D$ of vector fields tangent to
$E^{-1}\left(C\right).$ In \cite[Theorem 1.6]{Gomez}, Gomez-Mont
exhibited the existence of an exact sequence in cohomology written
\begin{equation}
\mathbb{H}^{1}\left(D,\mathbb{F}\right)\to H^{1}\left(D,\Theta\right)\to H^{1}\left(D,\textup{Hom}\left(\mathcal{\mathbb{F}},\frac{\Theta}{\mathcal{\mathbb{F}}}\right)\right).\label{exact.gomez}
\end{equation}

The space $\mathbb{H}^{1}\left(D,\mathcal{\mathbb{F}}\right)$ is
identified with the space of infinitemisal deformations of $\mathcal{F}$;
the space $H^{1}\left(D,\Theta\right)$ is identified with the space
of infinitesimal deformations of $C.$ Now, the sheaf $\mathcal{\mathbb{F}}$
is locally free of rank $1$. Thus, a section $\alpha$ of $\textup{Hom}\left(\mathcal{\mathbb{F}},\frac{\Theta}{\mathcal{\mathbb{F}}}\right)$
is completely determined by the image of $E^{\star}X$ or, equivalently
by the image of $\mathfrak{X}$. By contruction, $\mathcal{F}$ is
of \emph{second kind} as defined in \cite{Genzmer1}. The relations
established in \cite[Lemme 2.1]{Genzmer1} are written in our context
\[
\nu_{s}\left(i_{E^{\star}X}\Omega\right)=\nu_{s}\left(E^{\star}f\right)+\begin{cases}
1 & \textup{if }s\textup{ is invariant by }E^{\star}X\\
0 & \text{\textup{if not}}
\end{cases}
\]

where $i$ stands for the inner product. It can be seen that, as a
consequence, the morphism of sheaves defined by 
\[
\textup{Hom}\left(\mathcal{\mathbb{F}},\frac{\Theta}{\mathcal{\mathbb{F}}}\right)\to\Omega^{2}\left(D_{\mathfrak{X}}\right),\ \alpha\mapsto i_{\alpha\left(\mathfrak{X}\right)}\ensuremath{\Omega}\wedge i_{\mathfrak{X}}\ensuremath{\Omega}
\]
is an isomorphism of sheaves : here, $\Omega^{2}\left(D_{\mathfrak{X}}\right)$
is the sheaf over $D$ of $2-$forms $\eta$ for which the divisor
$\left(\eta\right)=\left(\eta=0\right)-\left(\eta=\infty\right)$
satisfies
\[
\left(\eta\right)\ge-D_{\mathfrak{X}}.
\]
Moreover, in \cite[Proposition 18]{YoyoBMS}, it is proved that, provided
that the relations (\ref{eq:valuation.saito}) are satisfied, we have
\[
\textup{H}^{1}\left(D,\Omega^{2}\left(D_{\mathfrak{X}}\right)\right)=0.
\]
Thus, the exact sequence (\ref{exact.gomez}) reduces to 
\begin{equation}
\mathbb{H}^{1}\left(D,\mathcal{F}\right)\to\textup{H}^{1}\left(D,\Theta\right)\to0.\label{eq:onto.defo}
\end{equation}

Since $D$ is of dimension $1$ and can be covered with Stein open
sets with no $3$ by $3$ intersections \cite[Definition 1.2.4]{SiuThm,univ},
the coherence of the involved sheaves ensures that for any $k\geq2$,
we have
\[
H^{k}\left(D,\Theta\right)=H^{k}\left(D,\mathcal{\mathbb{F}}\right)=0.
\]

Now consider an analytical deformation $C_{t},\ t\in\left(\mathbb{C},0\right)$
with $C_{0}=C$ such that for all $t\neq0$, $C_{t}$ belongs to \emph{the
generic component} of its moduli space, as defined in \cite[Theorem 3]{genzmer2020saito}.
Using some techniques similar as these developed in \cite[Proposition 2.2]{MR2422017}
for unfoldings of foliations and adapted to general deformations and
starting from the exact sequence (\ref{eq:onto.defo}), we obtain
a deformation of foliations $\mathcal{F}_{t},\ t\in\left(\mathbb{C},0\right),\ \mathcal{F}_{0}=\mathcal{F}$
such that for any $t$, $\mathcal{F}_{t}$ leaves invariant $C_{t}$.
This deformation is \emph{weakly equisingular}, in the sense that,
the family of numbered dual trees of the family of processes of desingularizations
of $\mathcal{F}_{t}$ is constant equal to the numbered tree $\AA$.
In particular, the valuations $\nu\left(\left(\mathcal{F}_{t}\right)_{s}^{E}\right)$
at any $s\in\AA$ satisfy ,
\[
\forall t\in\left(\mathbb{C},0\right),\ \nu\left(\left(\mathcal{F}_{t}\right)_{s}^{E}\right)=\nu\left(\mathcal{F}_{s}^{E}\right).
\]
Now, if $c$ is generic in the moduli space then, according to \cite[Theorem 4]{genzmer2020saito},
for any $X$ in $\textup{Der}\left(\log c\right)$, the following
lower bound holds 
\[
\forall s\in\AA,\ \nu\left(X_{s}^{E}\right)\geq\frac{\nu\left(c_{s}^{E}\right)+\delta_{s}^{X}}{2}-\pair{1-\Delta_{s}^{X}}{\frac{1}{2}}{\nu\left(c_{s}^{E}\right)+\delta_{s}^{X}}.
\]
Thus, for any $t\neq0,$ the foliation $\mathcal{F}_{t}$ - or a vector
field $X_{t}$ defining $\mathcal{F}_{t}$ - leaves invariant a curve
$C_{t}$ topologically equivalent to $C_{0}=C$ that is optimal for
$C_{t}$ and keeps on being optimal along the desingularization process
of $C_{t},$ that means precisely, is Saito for $C_{t}$.
\end{proof}

\subsection{Generic dimension of the moduli space $\textup{Mod}\left(C\right)$}

According to \cite[Theorem 4.2]{MatQuasi}, the generic dimension
of the moduli space of $C$ is equal to $\dim\textup{H}^{1}\left(D,\Theta\right)$
when $C$ is chosen generic in its moduli space. The results of this
section and these of \cite{genzmer2020saito} ensure that this dimension
can be computed from the topological datas associated to a Saito foliation
for a generic element in the topological class of $C.$ In \cite{moduligenz},
an algorithm is given to compute this topological datas when $C$
is an union of regular curve. This article implies that the same algorithm
still provides this topological datas in the general case. We implemented,
among other procedures this algorithm on Sage 9.{*} - or Python 3
-. See the routine \emph{Courbes.Planes} following the link
\begin{quote}
https://perso.math.univ-toulouse.fr/genzmer/
\end{quote}

\section{Examples}
\begin{example}[\emph{The Saito foliation of the double cusp}]
 The double cusp is the curve $C$ defined by 
\[
\left(y^{2}+x^{3}\right)\left(y^{2}-x^{3}\right)=0.
\]
It is a curve with $0-$dimensional moduli space. Its desingularization
$E$ consists in five elementary blowing-ups
\[
E=\bigcirc_{i=0}^{4}E_{i}.
\]
The Saito dicriticity of $C$ is given in Figure \ref{fig:Saito-dicriticities-of}.
The number on each vertex allows us to identify the order on the tree
defined by 
\[
0\leq1,\ 0\leq2,\ 1\leq4,\ 2\leq3.
\]
The dots encode the configuration. Here, the configuration associated
to the Saito dicriticity is 
\[
\epsilon_{0}=\epsilon_{1}=\epsilon_{2}=1,\ \epsilon_{3}=\epsilon_{4}=0.
\]
\begin{figure}
\begin{centering}
\includegraphics[scale=0.7]{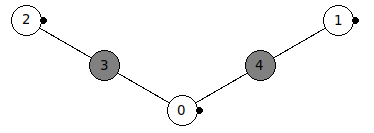}
\par\end{centering}
\caption{\label{fig:Saito-dicriticities-of}Saito dicriticities of the double
cusp.}
\end{figure}

It can be seen, by computing its desingularization, that the vector
field $X$ defined by 
\begin{align*}
X & =\left(\frac{9}{5}x^{3}y-x^{2}y^{2}+y^{3}-\frac{4}{5}x^{2}+xy\right)\partial_{x}\\
 & +\left(\frac{6}{5}x^{2}y^{2}-\frac{3}{2}xy^{3}-\frac{5}{6}x^{3}-\frac{6}{5}xy+\frac{2}{3}y^{2}\right)\partial_{y}
\end{align*}
 is Saito for the double cusp. Indeed, it is tangent to $C$ and non
dicritical. Its valuation satifies 
\[
2=\frac{\nu\left(C\right)=4}{2}-\pair{1-1}{\frac{1}{2}}4.
\]
After one blowing-up, it has three singularities along the exceptional
divisor given in the coordinates $\left(y=y_{1},x=y_{1}x_{1}\right)$
by 
\[
s_{1}=\left(0,0\right),\ s_{3}=\left(0,-\frac{6}{5}\right)\text{\textup{ and }}s_{2}=\left(0,\infty\right).
\]
The singularity $s_{3}$ is reduced : the quotient of the eigenvalues
of $E_{0}^{\star}X$ at $s_{2}$ is actually equal to $5.$ At $s_{1}$
and $s_{3}$, $E_{0}^{\star}X$ is of valuation $1$ which satisfies
\[
1=\frac{\nu\left(C_{s_{1}}^{E}\right)+1}{2}-\pair{1-1}{\frac{1}{2}}2=\frac{\nu\left(C_{s_{3}}^{E}\right)+1}{2}-\pair{1-1}{\frac{1}{2}}2.
\]
After blowing-up $s_{1}$, the vector field $\left(E_{0}\circ E_{1}\right)^{\star}X$
has two singularities along the new exceptional divisor. One is reduced
with positive and rational quotient of the eigenvalues. The other
is radial, that is, its linear part is locally in coordinates written
$x\partial_{x}+y\partial_{y}.$ The same occurs at $s_{3}$. At the
radial singularities $s_{5}$ and $s_{6}$ which are dicritical, one
has 
\[
1=\frac{\nu\left(C_{s_{5}\textup{ or }s_{6}}^{E}\right)+2}{2}-\pair{1-0}{\frac{1}{2}}3.
\]
As a consequence, $X$ is indeed Saito for $C.$ 
\end{example}

\begin{example}[\emph{Generic dimension of the moduli space of the union of $r$ cusps
of type $\left(2,3\right)$}]
In \cite{hernandes2023analytic}, the authors give a formula for
the generic dimension of the moduli space of the curve 
\[
\mathcal{C}_{r}=\left\{ \prod_{i=1}^{r}\left(y^{2}+a_{i}x^{3}\right)=0\right\} 
\]
 where $a_{i}\neq a_{j}\neq0$ for $i\neq j$. When $r$ is even,
this dimension happens to be equal to 
\begin{equation}
\frac{\left(r-1\right)\left(3r-5\right)+1}{2}.\label{eq:dimension.cusps}
\end{equation}

Let us illustrate how our algorithm works in this situation. The proximity
matrix of the $\mathcal{C}_{r}$ is 
\[
\left(\begin{array}{ccc}
1 & -1 & -1\\
0 & 1 & -1\\
0 & 0 & 1
\end{array}\right)
\]
and the numbering of $\AA$ is $\left(0,0,r\right).$ The Saito dicriticity
is equal to $\left(1,1,0\right)$ and the associated configuration
is $\left(2,1,\frac{r}{2}\right).$ After one blowing-up, according
to \cite[Proposition 4]{genzmer2020saito}, we get
\[
\dim H^{1}\left(D_{1},\Theta\right)=\frac{\left(r-1\right)\left(r-2\right)}{2}+\frac{\left(r-1\right)\left(r-2\right)}{2}=\left(r-1\right)\left(r-2\right).
\]
Now, after one blowing-up the curve $D_{1}\cup\mathcal{C}_{r}^{E_{1}}$
is given in local coordinates $y=y_{1}x_{1}$ by 
\[
x_{1}\prod_{i=1}^{r}\left(y_{1}^{2}+a_{i}x_{1}\right)=0.
\]
The proximity matrix of the desingularization of the latter curve
is now 
\[
\left(\begin{array}{cc}
1 & -1\\
0 & 1
\end{array}\right)
\]
and the numbering $\left(0,r+1\right).$ The Saito dicriticity is
equal to $\left(1,0\right)$ and the associated configuration is $\left(1,\frac{r}{2}\right).$
Thus, we obtain
\[
\dim H^{1}\left(D_{2},\Theta\right)=\frac{\left(\frac{r}{2}-1\right)\left(\frac{r}{2}-2\right)}{2}+\frac{\frac{r}{2}\left(\frac{r}{2}-1\right)}{2}.
\]
Finally, after one more blowing-up the curve $D_{1}\cup D_{2}\cup\mathcal{C}_{r}^{E_{2}}$
is given by 
\[
x_{2}y_{2}\prod_{i=1}^{r}\left(y_{2}+a_{i}x_{2}\right).
\]
The proximity matrix reduces to $\left(1\right)$ and the numbering
to $\left(r+2\right).$ The Saito dicriticity is just $\left(0\right)$
and the configuration $\left(\frac{r}{2}+1\right).$ Thus, still following
\cite[Proposition 4]{genzmer2020saito}, one has
\[
\dim H^{1}\left(D_{3},\Theta\right)=\frac{\left(\frac{r}{2}-1\right)\left(\frac{r}{2}-2\right)}{2}+\frac{\frac{r}{2}\left(\frac{r}{2}-1\right)}{2}+r-1.
\]
Adding the above dimensions leads to 
\begin{align*}
\dim H^{1}\left(D,\Theta\right) & =\left(r-1\right)\left(r-2\right)+\frac{\left(\frac{r}{2}-1\right)\left(\frac{r}{2}-2\right)}{2}+\frac{\frac{r}{2}\left(\frac{r}{2}-1\right)}{2}\\
 & +\frac{\left(\frac{r}{2}-1\right)\left(\frac{r}{2}-2\right)}{2}+\frac{\frac{r}{2}\left(\frac{r}{2}-1\right)}{2}+r-1\\
 & =\frac{\left(r-1\right)\left(3r-5\right)+1}{2}
\end{align*}
which is confirmed by the formula (\ref{eq:dimension.cusps}).
\end{example}

\begin{example}[\emph{Generic Tjurina number of a curve.}]
 The algorithm defined above allows us to provide immediately a computation
of the generic Tjurina number, that is, the dimension of the quotient
of $\mathbb{C}\left\{ x,y\right\} $ by the Tjurina ideal of $C$,
i.e $\left(f,\partial_{x}f,\partial_{y}f\right)$ where $f$ is an
equation of $C.$ Let $E$ be the desingularization process of $C.$
Above the exceptional divisor $D$ of $E$, we consider the sheaves
$T_{df}$ and $T_{f}$ of vector fields tangent respectively to the
foliation $E^{\star}df$ and $E^{-1}\left(f^{-1}\left(0\right)\right)$.
The following sequence of sheaves 
\[
0\to T_{df}\to T_{f}\xrightarrow{E^{\star}df\left(\cdot\right)}\left(f\circ E\right)\mathcal{O}_{D}\to0
\]
is exact \cite{JMSE}. The associated long exact sequence in cohomology
is written 
\begin{align*}
0 & \to H^{0}\left(D,T_{df}\right)\to H^{0}\left(D,T_{f}\right)\to H^{0}\left(D,\left(f\circ E\right)\mathcal{O}_{D}\right)\\
 & \to H^{1}\left(D,T_{df}\right)\to H^{1}\left(D,T_{f}\right)\to0
\end{align*}
since $H^{1}\left(D,\left(f\circ E\right)\mathcal{O}_{D}\right)=0.$
Now, we can identify the global sections of the above sheaves : 
\begin{align*}
H^{0}\left(D,\left(f\circ E\right)\mathcal{O}_{D}\right) & =\left(f\right)\\
H^{0}\left(D,T_{f}\right) & =\left\{ \left.X\textup{ vector field}\right|X\cdot f\in\left(f\right)\right\} 
\end{align*}
Therefore, the following sequence is exact
\[
0\to\frac{\left(f\right)}{\left\{ \left.X\cdot f\right|X\textup{ tangent to }f=0\right\} }\to H^{1}\left(D,T_{df}\right)\to H^{1}\left(D,T_{f}\right)\to0.
\]

Now, it can be seen that 
\[
\frac{\left(f\right)}{\left\{ \left.X\cdot f\right|X\textup{ tangent to }f=0\right\} }\simeq\frac{\left(f,\textup{Jac}f\right)}{\textup{Jac}f}.
\]
The previous short exact sequence ensures that 
\[
\dim H^{1}\left(D,T_{f}\right)-\dim H^{1}\left(D,T_{df}\right)+\dim\frac{\left(f,\textup{Jac}f\right)}{\textup{Jac}f}=0
\]
which can be also written 
\[
\tau\left(C\right)=\mu\left(C\right)-\delta\left(C\right)+\dim H^{1}\left(D,T_{f}\right)
\]

where $\mu\left(C\right)$ is the Milnor number of $C$ and $\delta\left(C\right)$
its modularity \cite{univ}. Now, if $C$ is chosen generic in its
moduli space, we obtain
\[
\tau_{\textup{gen}}\left(C\right)=\mu\left(C\right)-\delta\left(C\right)+\dim_{\textup{gen}}\mathbb{M}\left(C\right).
\]
Since the Milnor number and the modularity can be computed from the
numbered tree of $C$, the formula above yields an agorithm to compute
the generic Tjurina number of $C$ - which happens to be also the
minimal Tjurina number. 

As an example, the curve given by the following parametrization $C=\left(t^{9},t^{12}+t^{17}\right)$
has been studied by Peraire \cite{PERAIRE1997114} and she found 
\[
\tau_{\textup{gen}}\left(C\right)=80.
\]
It can be seen that 
\[
\mu\left(C\right)=98\textup{ and }\delta\left(C\right)=29.
\]

Our algorithm provides the generic dimension of the moduli space and
we find
\[
\dim_{\textup{gen}}\mathbb{M}\left(C\right)=11,
\]
which confirms the result of Peraire. 
\end{example}

\bibliographystyle{plain}
\bibliography{Bibliographie}

\end{document}